\documentclass[amsfonts,12pt]{amsart}
\usepackage{epsfig}
\usepackage{amsmath,amssymb}
\newtheorem{thm}{Theorem}[section]

\newtheorem{lem}[thm]{Lemma}
\newtheorem{cor}[thm]{Corollary}
\newtheorem{rem}[thm]{Remark}
\newtheorem{prp}[thm]{Proposition}
\newtheorem{ex}[thm]{Example}

\newcommand{\inte}{{\mathrm{int}}\,}

\newcommand{\cl}{{\mathrm{cl}}\,}
\newcommand{\conv}{{\mathrm{conv}}\,}

\newcommand{\ee}{\varepsilon}

\newcommand{\R}{\mathbb{R}}

\newcommand{\N}{\mathbb{N}}

\newcommand{\K}{\mathcal{K}}

\newcommand{\Cpm}{C_{\varphi,\mu}}
\newcommand{\Cppi}{C_{\varphi,\pi_K}}
\newcommand{\Cppip}{C_{\varphi,\pi_K^+}}
\newcommand{\Cpg}{C_{\varphi,\gamma_K}}
\newcommand{\Cpgp}{C_{\varphi,\gamma_K^+}}
\newcommand{\Cpmi}{C_{\varphi,\mu_i}}
\newcommand{\CpAm}{C_{\varphi,A\mu}}

\newcommand{\Cpmo}{C_{\varphi,\mu_0}}

\begin{document}
\hfill\today
\bigskip

\title[THE ORLICZ-BRUNN-MINKOWSKI THEORY]{THE ORLICZ-BRUNN-MINKOWSKI THEORY: A GENERAL FRAMEWORK, ADDITIONS, AND INEQUALITIES}

\author[Richard J. Gardner, Daniel Hug, and Wolfgang Weil]
{Richard J. Gardner, Daniel Hug, and Wolfgang Weil}
\address{Department of Mathematics, Western Washington University,
Bellingham, WA 98225-9063,USA} \email{Richard.Gardner@wwu.edu}
\address{Karlsruhe Institute of Technology, Department of Mathematics,
D-76128 Karlsruhe, Germany}
\email{daniel.hug@kit.edu}
\address{Karlsruhe Institute of Technology, Department of Mathematics,
D-76128 Karlsruhe, Germany} \email{wolfgang.weil@kit.edu}
\thanks{First author supported in
part by U.S.~National Science Foundation Grant DMS-1103612 and by the Alexander von Humboldt Foundation}
\pagestyle{myheadings}
\markboth{RICHARD J. GARDNER, DANIEL HUG, AND
WOLFGANG WEIL}{THE ORLICZ-BRUNN-MINKOWSKI THEORY}

\begin{abstract}
The Orlicz-Brunn-Minkowski theory, introduced by Lutwak, Yang, and Zhang, is a new extension of the classical Brunn-Minkowski theory.  It represents a generalization of the $L_p$-Brunn-Minkowski theory, analogous to the way that Orlicz spaces generalize $L_p$ spaces.  For appropriate convex functions $\varphi:[0,\infty)^m\to [0,\infty)$, a new way of combining arbitrary sets in $\R^n$ is introduced. This operation, called Orlicz addition and denoted by $+_{\varphi}$, has several desirable properties, but is not associative unless it reduces to $L_p$ addition.  A general framework is introduced for the Orlicz-Brunn-Minkowski theory that includes both the new addition and previously introduced concepts, and makes clear for the first time the relation to Orlicz spaces and norms.  It is also shown that Orlicz addition is intimately related to a natural and fundamental generalization of Minkowski addition called $M$-addition.  The results obtained show, roughly speaking, that the Orlicz-Brunn-Minkowski theory is the most general possible based on an addition that retains all the basic geometrical properties enjoyed by the $L_p$-Brunn-Minkowski theory.

Inequalities of the Brunn-Minkowski type are obtained, both for $M$-addition and Orlicz addition. The new Orlicz-Brunn-Minkowski inequality implies the $L_p$-Brunn-Minkowski inequality.  New Orlicz-Minkowski inequalities are obtained that generalize the $L_p$-Minkowski inequality.  One of these has connections with the conjectured log-Brunn-Minkowski inequality of Lutwak, Yang, and Zhang, and in fact these two inequalities together are shown to split the classical Brunn-Minkowski inequality.
\end{abstract}

\maketitle

\section{Introduction}
Beginning in the late nineteenth century, the classical Brunn-Minkowski theory was developed by Minkowski, Blaschke, Aleksandrov, Fenchel, and others. Combining two concepts, volume and Minkowski addition, it became an extremely powerful tool in convex geometry with significant applications to various other areas of mathematics. Schneider's classic text \cite{Sch93} is an excellent survey and source of references.

During the last few decades, the core Brunn-Minkowski theory has been extended in several important ways. One of the two extensions that concern us is the $L_p$-Brunn-Minkowski theory, which blends volume and a different way of combining sets called $L_p$ addition, introduced by Firey in the 1960's.  Denoted by $+_p$, this is defined for $1\le p\le \infty$ by
\begin{equation}\label{introMink22}
h_{K+_pL}(x)^p=h_K(x)^p+h_L(x)^p,
\end{equation}
for all $x\in\R^n$ and compact convex sets $K$ and $L$ in $\R^n$ containing the origin, where the functions are the support functions of the sets involved.  (See Section~\ref{prelim} for unexplained terminology and notation.)  When $p=\infty$, (\ref{introMink22}) is interpreted as $h_{K+_{\infty}L}(x)=\max\{h_K(x),h_L(x)\}$, as is customary. When $p=1$, (\ref{introMink22}) defines ordinary Minkowski addition and then $K$ and $L$ need not contain the origin.  In the hands of first Lutwak in the 1990's, and then Lutwak, Yang, and Zhang, and many others, the $L_p$-Brunn-Minkowski theory has allowed many of the already potent sharp affine isoperimetric inequalities of the classical theory, as well as related analytic inequalities, to be strengthened.  It has also provided tools for attacks on major unsolved problems such as the slicing problem of Bourgain, and consolidated connections between convex geometry and information theory.  See, for example, \cite{DP}, \cite{HS1}, \cite{HS2}, \cite{LYZ1}, \cite{LYZ2}, \cite{LYZ3}, \cite{LYZ4}, \cite{LYZ6}, and \cite{Wer12}.

The other extension of interest here is the still more recent Orlicz-Brunn-Minkowski theory, initiated by Lutwak, Yang, and Zhang \cite{LYZ7}, \cite{LYZ8}. In these papers the fundamental notions of $L_p$ centroid body and $L_p$ projection body were extended to an Orlicz setting.  To say that this involves replacing the function $t^p$ by an arbitrary convex function $\varphi:[0,\infty)\to[0,\infty)$ with $\varphi (0)=0$ masks the difficulty of the task, one also present in the transition from $L_p$ spaces to Orlicz spaces \cite{KraR61}, \cite{RaoRen91}. Like the $L_p$ extension of the Brunn-Minkowski theory, the newer Orlicz extension requires considerable subtlety.  So far only a few other articles advance the theory, among which we mention \cite{HLYZ} and \cite{HuaH12}, which address the Orlicz version of the Minkowski problem.

One obstacle in the development of the Orlicz-Brunn-Minkowski theory appears to have been the lack of a notion corresponding to $L_p$ addition.  Perhaps one reason for this is that the most obvious definition of an Orlicz addition---obtained by simply replacing $t^p$ by $\varphi(t)$ throughout (\ref{introMink22})---turns out to yield nothing new, as we prove in Theorem~\ref{OrlDef2}. One contribution of the present paper is to correct this deficiency.  For simplicity, we shall consider only sums of two sets in this introduction.  We define the Orlicz sum $K+_\varphi L$ of compact convex sets $K$ and $L$ in $\R^n$ containing the origin, implicitly, by
\begin{equation}\label{Orldef22}
\varphi\left( \frac{h_K(x)}{h_{K+_\varphi L}(x)}, \frac{h_L(x)}{h_{K+_\varphi L}(x)}\right) = 1,
\end{equation}
for $x\in \R^n$, if $h_K(x)+h_L(x)>0$, and by $h_{K+_\varphi L}(x)=0$, if $h_K(x)=h_L(x)=0$.  Here $\varphi\in\Phi_2$, the set of convex functions $\varphi:[0,\infty)^2\to [0,\infty)$ that are increasing in each variable and satisfy $\varphi (0,0)=0$ and $\varphi(1,0)=\varphi(0,1)=1$.  Orlicz addition reduces to $L_p$ addition, $1\le p<\infty$, when $\varphi(x_1,x_2)=x_1^p + x_2^p$, or $L_{\infty}$ addition, when $\varphi(x_1,x_2)=\max\{x_1,x_2\}$.

In Theorem~\ref{Orthm1}, we show that this Orlicz addition $+_{\varphi}$ has several desirable properties.  For example, it is continuous in the Hausdorff metric,  $GL(n)$ covariant, and preserves the $o$-symmetry (origin symmetry) of sets. (See Section~\ref{Maddition} for definitions of properties of additions.) In \cite[Theorem~7.6 and Corollary~7.7]{GHW}, it was demonstrated that the first two properties alone force any addition between $o$-symmetric compact convex sets to be $M$-addition $\oplus_M$ for some 1-unconditional (that is, symmetric with respect to the coordinate axes) compact convex set $M$ in $\R^2$.  This means that $K+_\varphi L$ equals
$$
K\oplus_M L= \{ a x + b y : x\in K, y\in L, (a,b )\in M\},
$$
for all $o$-symmetric compact convex sets $K$ and $L$ in $\R^n$. However, unless it is already $L_p$ addition, $1\le p\le \infty$, the new Orlicz addition is not associative, as we show in Theorem~\ref{thmassoc}.  This is a consequence of new results proved here and of \cite[Theorem~7.9]{GHW}, which states that with three trivial exceptions, any operation between {\em $o$-symmetric} compact convex sets that is continuous in the Hausdorff metric,  $GL(n)$ covariant, and associative must be $L_p$ addition for some $1\le p\le\infty$.  Moreover, in Theorem~\ref{thmcomm}, we prove that Orlicz addition is commutative if and only if it is $L_{\infty}$ addition or $\varphi(x_1,x_2)=\varphi_0(x_1)+\varphi_0(x_2)$, for some $\varphi_0\in \Phi$, the set of convex functions $\varphi:[0,\infty)\to [0,\infty)$ that satisfy $\varphi(0)=0$ and $\varphi(1)=1$.

One of the main discoveries we make is that there is a surprisingly close relationship between Orlicz addition and $M$-addition.  For example, as operations between $o$-symmetric compact convex sets, they are essentially the same when $M$ is 1-unconditional.  More specifically, in this context, Theorem~\ref{Orthm2} implies that if $\varphi\in\Phi_2$, then $+_{\varphi}$ is $M$-addition for $M=J_{\varphi}^{\circ}$, the polar of the 1-unconditional convex body $J_\varphi=\{(\pm x_1,\pm x_2)\in [-1,1]^2: \varphi(|x_1|,|x_2|)\le 1\}$. Conversely, by Corollary~\ref{corinverseTh4.4}, if $M$ is a $1$-unconditional convex body in $\R^2$ that contains $(1,0)$ and $(0,1)$ in its boundary, then $\oplus_{M}$ is Orlicz addition $+_{\varphi}$ for some $\varphi\in \Phi_2$.  (The condition that $M$ contains $(1,0)$ and $(0,1)$ in its boundary can be removed by a slightly different choice for the class $\Phi_2$; see Remark~\ref{remOMequiv}.)  Analogous results are obtained for Orlicz and $M$-addition as operations between compact convex sets containing the origin.

In our view, these results, together with those in \cite{GHW}, shed considerable light on the nature of possible future extensions to the Brunn-Minkowski theory.  The classical theory and the $L_p$-Brunn-Minkowski theory arose from combining their respective additions and volume, with $L_p$ addition for $p>1$ giving up two features of Minkowski addition: the algebraic property of distributivity (see \cite[Theorem~7.1]{GHW}) and the geometric property of translation invariance.  Similarly, the Orlicz-Brunn-Minkowski theory can now also be seen, retroactively, as arising from combining volume and the new Orlicz addition, which in general also loses commutativity and associativity.  What other additions and corresponding extensions of the Brunn-Minkowski theory lie ahead?  Restricting to $o$-symmetric compact convex sets, we now see that any such further extension must be based on an addition that discards at least one of two fundamental assets: continuity in the Hausdorff metric and $GL(n)$ covariance.  Thus, roughly speaking, the Orlicz-Brunn-Minkowski theory is the most comprehensive possible that retains the amenity of these two geometrical properties.

Beyond this insight, the present paper makes two further contributions.  The first is in providing, for the first time, a general framework for the Orlicz-Brunn-Minkowski theory and at the same time clarifying its relation to Orlicz spaces.  In Section~4, we show that the classical notion of an Orlicz norm leads to a very general construction of a compact convex set $\Cpm$ depending on a function $\varphi\in \Phi_m$, the natural generalization of the class $\Phi_2$ to  functions from $[0,\infty)^m$ to $[0,\infty)$, $m\ge 2$, and a Borel measure $\mu$ in $\left({\K}^n_o\right)^m$, where ${\K}^n_o$ denotes the class of compact convex sets in $\R^n$ containing the origin.  Special cases include not only the Orlicz sum described above, but also the Orlicz projection bodies and Orlicz centroid bodies introduced in \cite{LYZ7} and \cite{LYZ8}.

The second of the above-mentioned contributions lies in establishing an array of new inequalities tied to the Orlicz-Brunn-Minkowski theory.  The first step is taken in Section~\ref{Extension}, which provides a definition of Orlicz addition between arbitrary sets in $\R^n$ consistent with that described above. Even for $L_p$ addition, it is far from obvious how such a generalization can be made.  In fact, this was carried out only recently by Lutwak, Yang, and Zhang \cite{LYZ}, who used it to extend Firey's $L_p$-Brunn-Minkowski inequality
\begin{equation}\label{LpBM22}
{V}(K+_pL)^{p/n} \ge {V}(K)^{p/n}+
{V}(L)^{p/n},
\end{equation}
where $p>1$ and $V$ denotes volume, to compact sets $K$ and $L$ in $\R^n$, proving also that if ${V}(K){V}(L)>0$, equality holds if and only if $K$ and $L$ are convex, contain the origin, and are dilatates of each other.  (Setting $p=1$ in (\ref{LpBM22}) yields the classical Brunn-Minkowski inequality, whose equality condition is different; see \cite[Section~4]{Gar02}.) Here we establish, in Section~\ref{Minequality}, both $M$-addition and Orlicz addition versions of the Brunn-Minkowski inequality for compact sets, including equality conditions.  For example, in Corollary~\ref{OBMnewer} we prove that if $\varphi\in \Phi_2$,
then
\begin{equation}\label{introMIz}
1\ge \varphi\left(\left(\frac{{V}(K)}
{{V}(K+_{\varphi}L)}\right)^{1/n},\left(\frac{{V}(L)}
{{V}(K+_{\varphi}L)}\right)^{1/n}\right),
\end{equation}
for all compact sets $K$ and $L$ in $\R^n$ with ${V}(K){V}(L)>0$.  Moreover, we show that when $\varphi$ is strictly convex, equality holds if and only if $K$ and $L$ are convex, contain the origin, and are dilatates of each other.  When $\varphi(x_1,x_2)=x_1^p+x_2^p$, this yields (\ref{LpBM22}) and its equality condition.  In fact (\ref{introMIz}), together with our results in Section~\ref{Extension}, essentially constitute an Orlicz extension of the results in \cite{LYZ}.

It is well known that the classical Brunn-Minkowski inequality for compact convex sets is equivalent to Minkowski's first inequality; see, for example, \cite[Section~5]{Gar02}.  The latter states that for compact convex sets $K$ and $L$ in $\R^n$,
\begin{equation}\label{M122}
V_1(K,L)\ge {V}(K)^{(n-1)/n}{V}(L)^{1/n}.
\end{equation}
The quantity $V_1(K,L)$ on the left in (\ref{M122}) is a special mixed volume equal to
\begin{equation}\label{FV22}
V_1(K,L)=\frac1n\lim_{\ee\to 0+}\frac{V(K+\ee L)-V(K)}{\ee}=\frac1n \int_{S^{n-1}} h_{L}(u)\,dS(K,u),
\end{equation}
where $S(K,\cdot)$ is the surface area measure of $K$.  The middle expression in (\ref{FV22}) is the first variation of the volume of $K$ with respect to $L$ and the right-hand side of (\ref{FV22}) is its integral representation.  The $L_p$-Brunn-Minkowski theory received its greatest single impetus when Lutwak \cite{L1} found the appropriate $L_p$ versions of (\ref{M122}) and (\ref{FV22}) and their ingredients.  By replacing Minkowski addition and scalar multiplication in (\ref{FV22}) by $L_p$ addition and its scalar multiplication ($\ee\cdot_p L=\ee^{1/p}L$), he showed that for $p>1$,
\begin{equation}\label{Lpvariation22}
V_p(K,L)\ge {V}(K)^{(n-p)/n}{V}(L)^{p/n},
\end{equation}
with equality if and only if $K$ and $L$ are dilatates or $L=\{o\}$, where
\begin{equation}\label{FV2233}
V_p(K,L)=\frac{p}{n}\lim_{\ee\to 0+}\frac{V(K+_p\,\ee\cdot_pL)-V(K)}{\ee}=\frac1n \int_{S^{n-1}} h_{L}(u)^ph_K(u)^{1-p}\,dS(K,u).
\end{equation}
Here $K$ is a convex body containing the origin in its interior and $L$ is a compact convex set containing the origin, assumptions we shall retain for the remainder of this introduction.

In particular, in the $L_p$-Brunn-Minkowski theory, $S(K,\cdot)$ is replaced by the $p$-surface area measure $S_p(K,\cdot)$ given by
$$
dS_p(K,\cdot)=h_K^{1-p}\,dS(K,\cdot).
$$
Haberl and Parapatits \cite{HabP13+} provide a characterization of these measures.

In Section~\ref{OMImv} we introduce a new notion of Orlicz linear combination, by means of an appropriate modification of (\ref{Orldef22}) (see (\ref{OrldefGG}) below).  Unlike the $L_p$ case, an Orlicz scalar multiplication cannot generally be considered separately.  The particular instance of interest corresponds to using (\ref{Orldef22}) with $\varphi(x_1,x_2)=\varphi_1(x_1)+\ee\varphi_2(x_2)$ for $\ee>0$ and some $\varphi_1,\varphi_2\in \Phi$, in which case we write $K+_{\varphi,\ee}L$ instead of $K+_{\varphi}L$.  If $\varphi_1(t)=\varphi_2(t)=t^p$, then $K+_{\varphi,\ee}L=K+_p\,\ee\cdot_pL$, as in (\ref{FV2233}).  In Theorem~\ref{OrlMV}, we compute the Orlicz first variation of volume, obtaining the equation
\begin{equation}\label{OrlMVdefff22}
\frac{(\varphi_1)'_l(1)}{n}\lim_{\ee\rightarrow 0+}\frac{V(K+_{\varphi,\ee}L)-V(K)}{\ee}=\frac{1}{n}\int_{S^{n-1}}
\varphi_2\left(\frac{h_L(u)}{h_K(u)}\right)h_K(u)\,dS(K,u).
\end{equation}
Denoting by $V_{\varphi}(K,L)$, for any $\varphi\in \Phi$, the integral on the right-hand side of (\ref{OrlMVdefff22}) with $\varphi_2$ replaced by $\varphi$, we see that either side of the equation (\ref{OrlMVdefff22}) is equal to $V_{\varphi_2}(K,L)$ and therefore this new Orlicz mixed volume plays the same role as $V_p(K,L)$ in the $L_p$-Brunn-Minkowski theory. In Theorem~\ref{NewMIthm}, we establish the following Orlicz-Minkowski inequality:
\begin{equation}\label{NewMI22}
V_{\varphi}(K,L)\ge V(K)\varphi\left(\left(\frac{V(L)}{V(K)}\right)^{1/n}\right).
\end{equation}
Here, if $\varphi$ is strictly convex, equality holds if and only if $K$ and $L$ are dilatates or $L=\{o\}$.
Note that when $\varphi(t)=t^p$, $p\ge 1$, (\ref{NewMI22}) and (\ref{OrlMVdefff22}) become (\ref{Lpvariation22}) and (\ref{FV2233}), respectively.

Other approaches are possible and this is one topic in Section~\ref{Alternative}.  A different Orlicz version of Minkowski's first inequality (\ref{M122}) is presented in Theorem~\ref{FVnew}.  This results from replacing the left-hand side of (\ref{M122}) by the quantity
$$
\inf\left\{\lambda>0:\int_{S^{n-1}}\varphi
\left(\frac{h_L(u)}{\lambda h_K(u)}\right)h_K(u)\,dS(K,u)\le nV(K)\right\},
$$
an idea suggested by the definition of Orlicz projection bodies in \cite{LYZ7}.

Finally, we discover an intriguing connection between (\ref{NewMI22}) and the log-Minkowski inequality
\begin{equation}\label{logMink22}
\int_{S^{n-1}}h_K(v)\log\left(\frac{h_L(v)}{h_K(v)}\right)
\,dS(K,v)\ge V(K)\log\left(\frac{V(L)}{V(K)}\right),
\end{equation}
for $o$-symmetric convex bodies $K$ and $L$, proved by B\"{o}r\"{o}czky, Lutwak, Yang, and Zhang \cite{BLYZ} when $n=2$ and conjectured by them to hold for all $n\ge 2$.  For such bodies, this is stronger than Minkowski's first inequality (\ref{M122}).  Inequality (\ref{logMink22}) is just (\ref{NewMI22}) with $\varphi(t)=\log t$, but of course since $\log t$ is concave, not convex, this choice is invalid.  On the other hand, Lemma~\ref{Dan} shows that it is possible to take $\varphi(t)=-\log(1-t)$ in (\ref{NewMI22}) provided $L\subset\inte K$.  In Theorem~\ref{Wolfgang1}, this is used to obtain the inequality
\begin{equation}\label{cons122}
\log\left(\frac{V(K)^{1/n}-V(L)^{1/n}}{V(K)^{1/n}}\right)\ge \frac{1}{nV(K)}\int_{S^{n-1}}h_{K}(u)\log\left(\frac{h_K(u)-h_L(u)}{h_{K}(u)}\right)\,
dS(K,u),
\end{equation}
with equality if and only if $K$ and $L$ are dilatates or $L=\{o\}$.
We conclude the paper by showing that (\ref{logMink22}) and (\ref{cons122}) can be used together to split the classical Brunn-Minkowski inequality.  This means that $A\ge B\ge C$, where $A\ge C$ is the Brunn-Minkowski inequality, (\ref{logMink22}) implies $A\ge B$, and (\ref{cons122}) implies $B\ge C$.

\section{Definitions and preliminaries}\label{prelim}

As usual, $S^{n-1}$ denotes the unit sphere and $o$ the origin in Euclidean
$n$-space $\R^n$.  We shall assume that $n\ge 2$ throughout.  The unit ball in $\R^n$ will be denoted by $B^n$. The
standard orthonormal basis for $\R^n$ will be $\{e_1,\dots,e_n\}$.  Otherwise, we usually denote the coordinates of $x\in \R^n$ by $x_1,\dots,x_n$.  If $x,y\in\R^n$, then $x\cdot y$ is the inner product of $x$ and $y$, and
$[x,y]$ is the line segment with endpoints $x$ and $y$. If $x\in \R^n\setminus\{o\}$, then $x^{\perp}$ is the $(n-1)$-dimensional subspace orthogonal to $x$ and $l_x$ is the line through $o$ containing $x$.  (Throughout the paper, the term {\em subspace} means a linear subspace.)

If $X$ is a set in $\R^n$,  we denote by $\partial X$, $\cl X$, $\inte X$, and $\conv X$ the {\it boundary}, {\it closure}, {\it interior}, and {\it convex hull} of $X$, respectively.  If $S$ is a subspace of $\R^n$, then $X|S$ is the (orthogonal) projection of $X$ on $S$ and $x|S$ is the projection of a vector $x\in\R^n$ on $S$.

If $t\in\R$, then $tX=\{tx:x\in X\}$. When $t>0$, $tX$ is called a {\em dilatate} of $X$.

A {\it body} in $\R^n$ is a compact set equal to the closure of its interior.

For a compact set $K\subset\R^n$, we write $V(K)$ for the ($n$-dimensional) Lebesgue measure of $K$ and call this the {\em volume} of $K$. We follow Schneider \cite{Sch93} by writing $\kappa_n$ for the volume $V(B^n)$ of
the unit ball in $\R^n$, so that $\kappa_n=\pi^{n/2}/\,\Gamma(1+n/2)$.

A subset of $\R^n$ is {\it $o$-symmetric} if it is centrally symmetric, with center at the
origin.  We shall call a set in $\R^n$ {\em $1$-unconditional} if it is symmetric with respect to each coordinate hyperplane; this is traditional in convex geometry for compact convex sets.

Let ${\mathcal K}^n$ be the class of nonempty compact convex
subsets of  $\R^n$, let ${\mathcal K}^n_s$ denote the class of $o$-symmetric members of ${\mathcal K}^n$, let ${\mathcal{K}}_o^n$ be the class of members of ${\mathcal K}^n$ containing the origin, and let ${\mathcal{K}}_{oo}^n$ be those sets in ${\mathcal K}^n$ containing the origin in their interiors. A set $K\in {\mathcal K}^n$ is called a {\em
convex body} if its interior is nonempty.

If $K\subset\R^n$ is a nonempty closed (not necessarily bounded) convex set, then
$$
h_K(x)=\sup\{x\cdot y: y\in K\},
$$
for $x\in\R^n$, defines the {\it support function} $h_K$ of $K$. A nonempty closed convex set is uniquely determined by its support function.   Support functions are {\em homogeneous of degree 1}, that is,
$$h_K(rx)=rh_K(x),$$
for all $x\in \R^n$ and $r\ge 0$, and are therefore often regarded as functions on $S^{n-1}$.  They are also {\em subadditive}, i.e.,
$$h_K(x+y)\le h_K(x)+h_K(y),$$
for all $x,y\in \R^n$.  Any real-valued function on $\R^n$ that is {\em sublinear}, that is, both homogeneous of degree 1 and subadditive, is the support function of a unique compact convex set.  If $x\in\R^n\setminus\{o\}$, then
\begin{equation}\label{suppset}
F(K,x)=\{y\in \R^n: x\cdot y=h_K(x)\}\cap K
\end{equation}
is the {\em support set} of $K$ with outer normal vector $x$.  Proofs of these facts can be found in \cite{Sch93}.

The {\em surface area measure} of a compact convex set $K$ in $\R^n$ is a Borel measure in $S^{n-1}$, denoted by $S(K,\cdot)$. The centroid of $S(K,\cdot)$ is the origin.  The special mixed volume $V_1(K,L)$ of $K, L\in {\mathcal K}^n$ given by
\begin{equation}\label{V1form}
V_1(K,L)=\frac1n \int_{S^{n-1}} h_{L}(u)\,dS(K,u)=\frac1n \lim_{\ee\to 0+}\frac
{V(K+\ee L)-V(K)}{\ee}
\end{equation}
satisfies $V_1(K,K)=V(K)$. {\em Minkowski's first inequality} \cite[Theorem~B.2.1]{Gar06}, \cite[Theorem~6.2.1]{Sch93} states that
\begin{equation}\label{MinkIneq}
V_1(K,L)^{n}\ge V(K)^{n-1}V(L),
\end{equation}
with equality if and only if $K$ and $L$ lie in parallel hyperplanes or are homothetic.

The equation
\begin{equation}\label{normconem}
d\overline{V}_n(K,v)=\frac{h_K(v)}{nV(K)}\,dS(K,v)
\end{equation}
defines a Borel measure in $S^{n-1}$, the {\em normalized cone measure}, for $K\in {\mathcal{K}}_{oo}^n$.  By (\ref{V1form}) with $K=L$, $\overline{V}_n(K,\cdot)$ is a probability measure in $S^{n-1}$.

Let $K\subset\R^n$ be a nonempty closed convex set. If $S$ is a subspace of $\R^n$, then it is easy to show that
\begin{equation}\label{hproj}
h_{K|S}(x)=h_K(x|S),
\end{equation}
for each $x\in \R^n$.  The formula (see \cite[(0.27), p.~18]{Gar06})
\begin{equation}\label{suppchange}
h_{AK}(x)=h_K(A^t x),
\end{equation}
for $x\in \R^n$ and a linear transformation $A:\R^n\to\R^n$, gives the change in a support function under $A$, where $A^t$ denotes
the transpose of $A$.  (Equation (\ref{suppchange}) is proved in \cite[p.~18]{Gar06} for compact sets and $A\in GL(n)$, but the proof is the same if $K$ is unbounded or $A$ is singular.)

The {\em polar set} of an arbitrary set $K$ in $\R^n$ is
$$
K^{\circ}=\{x\in \R^n:\,x\cdot y\le 1 {\mathrm{~for~all~}} y\in K\}.$$
See, for example, \cite[p.~99]{Web}.

Recall that $l_x$ is the line through the origin containing $x\in\R^n\setminus\{o\}$. A set $L$ in $\R^n$ with $o\in L$ is {\it star-shaped at $o$} if $L\cap l_u$ is a (possibly degenerate) closed line segment for each $u\in
S^{n-1}$. If $o\in L$ and $L$ is star-shaped at $o$, we define its {\it radial
function} $\rho_L$ for $x\in \R^n\setminus\{o\}$ by
\begin{equation}\label{radialfunction}
\rho_{L}(x)=\max\{\lambda:\lambda x\in L\}.
\end{equation}
See \cite[Section~0.7]{Gar06}, where these terms are defined more generally, however.

The vector or Minkowski sum of sets $X$ and $Y$ in $\R^n$ is defined by
$$
X+Y=\{x+y: x\in X, y\in Y\}.
$$
When $K,L\in {\mathcal K}^n$, $K+L$ can be equivalently defined as the compact convex set such that
$$
h_{K+L}(u)=h_K(u)+h_L(u),
$$
for all $u\in S^{n-1}$.

Let $1< p\le \infty$. Firey \cite{Fir61}, \cite{Fir62} introduced the notion of what is now called the {\em $L_p$ sum} of $K,L\in {\mathcal K}^n_o$.  (The operation has also been called Firey addition, as in \cite[Section~24.6]{BZ}.)  This is the compact convex set $K+_pL$ defined by
\begin{equation}\label{Lpaddition}
h_{K+_pL}(u)^p=h_K(u)^p+h_L(u)^p,
\end{equation}
for $u\in S^{n-1}$ and $p<\infty$, and by
$$
h_{K+_{\infty}L}(u)=\max\{h_K(u),h_L(u)\},
$$
for all $u\in S^{n-1}$. Note that $K+_{\infty}L=\conv(K\cup L)$.

Lutwak, Yang, and Zhang \cite{LYZ} extended the previous definition for $1<p<\infty$, as follows.  Let $K$ and $L$ be arbitrary subsets of $\R^n$ and define
\begin{equation}\label{LYZLp}
K+_pL=\left\{(1-t)^{1/p'}x+t^{1/p'}y: x\in K, y\in L, 0\le t\le 1\right\},
\end{equation}
where $p'$ is the H\"{o}lder conjugate of $p$, i.e., $1/p+1/p'=1$.  In \cite{LYZ} it is shown that when $K,L\in {\mathcal K}^n_o$, this definition agrees with the previous one.  However, the right-hand side of (\ref{LYZLp}) is not always convex for $K, L\in {\mathcal K}^n$.  To see this, take $K=\{x\}$ and $L=\{y\}$, where $x$ and $y$ do not lie on the same line through the origin.  Then $K+_pL$ is a nonlinear curve that approaches $[x,x+y]\cup [y,x+y]$ as $p\rightarrow 1$ and $[x,y]$ as $p\rightarrow\infty$.

Another reasonable definition of the $L_p$ sum of $K,L\in {\mathcal K}^n$ for $1\le p\le\infty$ is given in \cite[Example~6.7]{GHW}.

The {\em left derivative} and {\em right derivative} of a real-valued function $f$ are denoted by $f'_l$ and $f'_r$, respectively.

Suppose that $\mu$ is a probability measure on a space $X$ and $g:X\to I\subset \R$ is a $\mu$-integrable function, where $I$ is a possibly infinite interval.  {\em Jensen's inequality} states that if $\varphi:I\to \R$ is a convex function, then
\begin{equation}\label{JenIneq}
\int_X \varphi\left(g(x)\right)\,d\mu(x)\ge \varphi\left(
\int_X g(x)\,d\mu(x)\right).
\end{equation}
If $\varphi$ is strictly convex, equality holds if and only if $g(x)$ is constant for $\mu$-almost all $x\in X$.  See, for example, \cite[Theorem~3.10, p.~165 and p.~243]{Hof94}.

Throughout the paper, $\Phi_m$, $m\in\N$, denotes the set of convex functions $\varphi: [0,\infty)^m\to [0,\infty)$ that are increasing in each variable and satisfy $\varphi (o)=0$ and $\varphi(e_j)=1>0$, $j=1,\dots,m$.  The normalization here is a matter of convenience and other choices are possible.  For example, the requirement $\varphi(e_j)=c>0$, $j=1,\dots,m$, for some constant $c$, can be adopted without affecting any of the results.  More generally still, it would be possible to work with the class $\overline{\Phi}_m$, defined as above but with the assumption that $\varphi(e_j)=1>0$, $j=1,\dots,m$, replaced by the requirement that $\varphi(te_j)$, $t\ge 0$, is not identically zero, for $j=1,\dots,m$, in which case the role of $e_j$ is replaced by the point $t_je_j$ satisfying $\varphi(t_je_j)=1$, for $j=1,\dots,m$.  Again, the results are essentially unaffected, though some adjustments would be required.  In this regard, see Remarks~\ref{remOMequiv} and~\ref{remLC}.

When $m=1$, we shall write $\Phi$ instead of $\Phi_1$.

If $\varphi\in \Phi$, we put
\begin{equation}\label{tau}
\tau=\tau(\varphi)=\max\{t\ge 0:\varphi(t)=0\}<1.
\end{equation}

\section{Properties of operations and $M$-addition}\label{Maddition}

Let $\mathcal{C}\subset\mathcal{D}$ be classes of sets in $\R^n$ and let $*:{\mathcal{C}}^m\rightarrow {\mathcal{D}}$ be an $m$-ary operation, with values denoted by $*(K_1,\dots,K_m)$ for $K_1,\dots,K_m\in \mathcal{C}$.  Here, and throughout the paper, we always assume that $m\ge 2$.  When $m=2$, we shall write $K*L$, for $K,L\in \mathcal{C}$, instead of $*(K,L)$.

In the following list, it is assumed that $\mathcal{C}$ and $\mathcal{D}$ are suitable classes for the property under consideration.  The properties are supposed to hold for all appropriate $K, L, N, K_j, L_j, K_{ij}\in{\mathcal{C}}$. See \cite{GHW} for further properties and information.

\medskip

1. ({\em Commutativity} ($m=2$ only))\quad $K*L=L*K$.

2. ({\em Associativity} ($m=2$ only))\quad $K*(L*N)=(K*L)*N$.

3. ({\em Homogeneity of degree $1$}) \quad $r\left(*(K_1,\dots,K_m)\right)=*(rK_1,\dots,rK_m)$, for all $r\ge 0$.

4.  ({\em Identity}) \quad $*(\{o\},\dots,\{o\},K_j,\{o\},\dots,\{o\})=K_j$, for $j=1,\dots,m$.

5. ({\em Continuity}) \quad $K_{ij}\rightarrow K_{0j},\, j=1,\dots,m \Rightarrow *(K_{i1},\dots,K_{im})\rightarrow *(K_{01},\dots,K_{0m})$ as $i\rightarrow\infty$ in the Hausdorff metric.

6. ({\em $GL(n)$ covariance}) \quad $A\left(*(K_1,\dots,K_m)\right)=*(AK_1,\dots,AK_m)$, for all $A\in GL(n)$.

7. ({\em Projection covariance}) \quad $\left(*(K_1,\dots,K_m)\right)|S=*(K_1|S,\dots,K_m|S)$ for every subspace $S$ of $\R^n$.

8. ({\em Monotonicity}) \quad $K_j\subset L_j$, $j=1,\dots,m$ $\Rightarrow *(K_1,\dots,K_m)\subset *(L_1,\dots,L_m)$.

\medskip

A straightforward modification of the proof of \cite[Lemma~4.1]{GHW} yields the following useful result.

\begin{lem}\label{CGimpliesPm}
Let ${\mathcal{C}}\subset{\mathcal{K}}^n$ be closed under the action of $GL(n)$ and the taking of Hausdorff limits. If $*:{\mathcal{C}}^m\rightarrow {\mathcal{K}}^n$ is continuous and $GL(n)$ covariant, then it is also projection covariant.
\end{lem}

Let $M$ be an arbitrary subset of $\R^m$. The {\em $M$-sum} $\oplus_M (K_1,K_2,\dots, K_m)$ of arbitrary sets $K_1,K_2,\dots,K_m$ in $\R^n$ is defined in \cite[Section~6]{GHW} by
\begin{equation}\label{Mndef}
\oplus_M (K_1,K_2,\dots, K_m)= \left\{\sum_{j=1}^ma_jx^{(j)}: x^{(j)}\in K_j, (a_1,a_2,\dots,a_m)\in M\right\}.
\end{equation}
(When $m>2$, this was called an $M$-combination in \cite{GHW}.) An equivalent definition is
\begin{equation}\label{Mnaltdef}
\oplus_M (K_1,K_2,\dots, K_m)= \cup\left\{a_1K_1+a_2K_2+\cdots +a_mK_m : (a_1,a_2,\dots,a_m)\in M\right\}.
\end{equation}
Several properties of $M$-addition follow easily from these equivalent definitions.  The $m$-ary operation $\oplus_M$ is $GL(n)$ covariant, hence homogeneous of degree 1, and monotonic.  If $M$ and $K_j$, $j=1,\dots,m$, are compact, then $\oplus_M (K_1,K_2,\dots, K_m)$ is also compact.

When $m=2$, the $M$-sum of sets $K$ and $L$ is denoted by $K\oplus_M L$.  Note that in this case, if $M=\{(1,1)\}$, then $\oplus_M$ is ordinary vector or Minkowski addition, and if
$$
M=\left\{(a,b)\in [0,1]^2:a^{p'}+b^{p'}=1\right\}=\left\{\left((1-t)^{1/p'},t^{1/p'}\right): 0\le t\le 1\right\},
$$
where $p>1$ and $1/p+1/p'=1$, then $\oplus_M$ is $L_p$ addition as defined in \cite{LYZ}.
The limiting case $p=1$, $p'=\infty$ gives $M=[e_1,e_1+e_2]\cup [e_2,e_1+e_2]$, which corresponds to the operation
$$K\oplus_M L= \left(\conv\{K,o\}+L\right)\cup\left(K + \conv\{L,o\}\right).$$
The case $p=\infty$, $p'=1$ corresponds to $M=[e_1,e_2]$ and
$$K\oplus_M L=\{(1-t)x+ty:x\in K,y\in L, 0\le t\le 1\}=\conv(K\cup L).$$

It appears that $M$-addition was first introduced, for centrally symmetric compact convex sets $K$ and $L$ and a $1$-unconditional convex body $M$ in $\R^2$, by Protasov \cite{Pro97}, motivated by work on the joint spectral radius in the theory of normed algebras.  Protasov proved that if $M$ is a $1$-unconditional compact convex subset in $\R^2$, then $\oplus_M:\left({\mathcal{K}}^n_s\right)^2\rightarrow {\mathcal{K}}^n_s$.  (This proof is omitted in the English translation.)  In
\cite[Corollary~6.4]{GHW} this result is generalized to $M$-sums where $m\ge 2$.  Also, in \cite[Theorem~6.1(i) and Corollary~6.4]{GHW}, it is shown that if $M$ is a compact convex subset of $[0,\infty)^m$, then $\oplus_M:\left({\mathcal{K}}^n\right)^m\rightarrow {\mathcal{K}}^n$, and in this case we also have $\oplus_M:\left({\mathcal{K}}^n_s\right)^m\rightarrow {\mathcal{K}}^n_s$.  In either of these settings, $\oplus_M$ is continuous and hence, by Lemma~\ref{CGimpliesPm}, projection covariant.  Moreover, by \cite[Theorem~6.5(ii)]{GHW}, we have
\begin{equation}\label{nn}
h_{\oplus_M(K_1,\ldots,K_m)}(x)=h_{M}\left(h_{K_1}(x),\ldots,h_{K_m}(x)\right),
\end{equation}
for all $K_1,\ldots,K_m\in\K^n_s$ and all $x\in \R^n$.

By \cite[Theorem~7.6 and Corollary~7.7]{GHW}, an operation $*:\left({\mathcal{K}}^n_s\right)^2\rightarrow {\mathcal{K}}^n$ is continuous and $GL(n)$-covariant if and only if it is projection covariant, and such operations are precisely those defined for all $K,L\in {\mathcal{K}}^n_s$ by $K*L=K\oplus_M L$, where $M$ is a $1$-unconditional compact convex subset of $\R^2$. Furthermore, \cite[Theorem~7.9]{GHW} states that $*:\left({\mathcal{K}}^n_s\right)^2\rightarrow {\mathcal{K}}^n$ is projection covariant and associative if and only if $*=\oplus_M$, where $M=\{o\}$, or $M=[-e_1,e_1]$, or $M=[-e_2,e_2]$, or $M$ is the unit ball in $l^2_p$ for some $1\le p\le \infty$.  The latter case means that the addition is $L_p$ addition for some $1\le p\le \infty$.

It is also true, by \cite[Theorems~9.7 and~9.9]{GHW}, that an operation $*:\left({\mathcal{K}}^n\right)^2\rightarrow {\mathcal{K}}^n$ (or $*:\left({\mathcal{K}}^n_o\right)^2\rightarrow {\mathcal{K}}^n$) is continuous and $GL(n)$-covariant if and only if it is projection covariant, and the latter holds if and only if there is a nonempty closed convex set $M$ in $\R^4$ such that
$$
h_{K*L}(x)=h_{M}\left(h_K(-x),h_K(x),h_L(-x), h_L(x)\right),
$$
for all $K,L\in {\mathcal{K}}^n$ (or all $K,L\in {\mathcal{K}}^n_o$, respectively) and $x\in\R^n$.

We shall need generalizations of \cite[Theorems~7.6 and~9.7]{GHW} to $m$-ary operations, presented as Theorems~\ref{Thmnew7.6} and~\ref{thm275b} below.  In order to keep the exposition reasonably short, frequent but brief references are made to some arguments in \cite{GHW}, given there in full detail. We start with the following lemma.

\begin{lem}\label{generalLe7.4}
\noindent ${\mathrm{(i)}}$  The $m$-ary operation $*:(\K^n_s)^m\to\K^n$ is projection covariant if and only if there is a homogeneous of degree 1 function
$f:[0,\infty)^m\to[0,\infty)$ such that
\begin{equation}\label{geneq7.4}
h_{*(K_1,\dots,K_m)}(x)=f(h_{K_1}(x),\ldots,h_{K_m}(x)),
\end{equation}
for all $K_1,\ldots,K_m\in\K^n_s$ and all $x\in \R^n$.

\smallskip

\noindent ${\mathrm{(ii)}}$ The function $f:[0,\infty)^m\to[0,\infty)$ in \eqref{geneq7.4} is increasing in each variable.

\smallskip

\noindent ${\mathrm{(iii)}}$ If $*:(\K^n_s)^m\to\K^n$ is projection covariant, then $*:(\K^n_s)^m\to\K^n_s$.
\end{lem}

\begin{proof}
\noindent ${\mathrm{(i)}}$  Suppose that $*:(\K^n_s)^m\to\K^n$ is projection covariant. Let $u\in S^{n-1}$ and recall that $l_u$ denotes the line through the origin parallel to $u$.  For $K_1,\ldots,K_m\in\K^n_s$, we have
$$
\left(*(K_1,\dots,K_m)\right)|l_u=(K_1|l_u)*\dots*(K_m|l_u).
$$
Arguing as in the proof of \cite[Lemma~7.4]{GHW}, we obtain functions $f_u,g_u:[0,\infty)^m\to\R$ such that $-g_u\le f_u$ and
$$
*\left([-t_1 u,t_1u],\dots,[-t_m u,t_m u]\right)=[-g_u(t_1,\ldots,t_m)u,f_u(t_1,\ldots,t_m)u],
$$
for all $t_1,\ldots,t_m\ge 0$. As in the proof of \cite[Lemma~7.4]{GHW}, one shows that $f_u$ and $g_u$ are independent of $u$, nonnegative, and homogeneous of degree 1 and that therefore \eqref{geneq7.4} holds.

As in \cite[Lemma~7.4]{GHW}, the converse follows easily from (\ref{hproj}) and (\ref{geneq7.4}).

\noindent ${\mathrm{(ii)}}$ To show that $f$ is increasing in each variable, we can follow the proof of the ``only if" part of \cite[Lemma~6.8]{GHW}.  (This uses only \eqref{geneq7.4} and the fact that $f$ is homogeneous of degree 1.)

\noindent ${\mathrm{(iii)}}$ This is an easy consequence of (\ref{geneq7.4}).
\end{proof}

\begin{thm}\label{Thmnew7.6}
An $m$-ary operation $*:(\K^n_s)^m\to\K^n$ is projection covariant if and only if it can be defined
by
\begin{equation}\label{geneq7.6}
h_{*(K_1,\ldots,K_m)}(x)=h_M(h_{K_1}(x),\ldots,h_{K_m}(x)),
\end{equation}
for all $K_1,\ldots,K_m\in\K^n_s$ and all $x\in \R^n$, or equivalently by
\begin{equation}\label{geneq7.6b}
*(K_1,\ldots,K_m)=\oplus_M(K_1,\ldots,K_m),
\end{equation}
where $M$ is a 1-unconditional compact convex set in $\R^m$ uniquely determined by \eqref{geneq7.6}.
\end{thm}

\begin{proof}
By Lemma~\ref{generalLe7.4}(i), an operation defined by (\ref{geneq7.6}) is projection covariant.

If $*:(\K^n_s)^m\to\K^n$ is projection covariant, then by Lemma~\ref{generalLe7.4}, $*:(\K^n_s)^m\to\K^n_s$ and there is a homogeneous of degree 1 function $f:[0,\infty)^m\to[0,\infty)$, increasing in each variable, such that \eqref{geneq7.4}
holds for all $K_1,\ldots,K_m\in\K^n_s$ and all $x\in \R^n$.  Let
$x=(x_1,\dots,x_m)\in [0,\infty)^m$ and $y=(y_1,\dots,y_m)\in [0,\infty)^m$, and define $K_j=[-x_je_1-y_je_2,x_je_1+y_je_2]$ for $j=1,\ldots,m$. Then $K_j\in\K^n_s$, $h_{K_j}(e_1+e_2)=x_j+y_j$, $h_{K_j}(e_1)=x_j$, and $h_{K_j}(e_2)=y_j$. Applying \eqref{geneq7.4} and the subadditivity of support functions, we obtain
$$
f(x+y)=h_{*(K_1,\ldots,K_m)}(e_1+e_2)\le h_{*(K_1,\ldots,K_m)}(e_1)+h_{ *(K_1,\ldots,K_m)}(e_2)=f(x)+f(y).
$$
Thus $f$ is sublinear. Define a function $F:\R^m\to[0,\infty)$ by $F(x_1,\ldots,x_m)=f(|x_1|,\ldots,|x_m|)$. Then $F$ is homogeneous of
degree 1. Using the fact that $f$ is increasing in each argument and subadditive, we obtain
\begin{align*}
F(x+y)&=f(|x_1+y_1|,\ldots, |x_m+y_m|)\le f(|x_1|+|y_1|,\ldots, |x_m|+|y_m|)\\
&=f((|x_1| ,\ldots, |x_m| )+(|y_1| ,\ldots, |y_m| ))\le f(|x_1| ,\ldots, |x_m| )+f(|y_1| ,\ldots, |y_m| )\\
&=F(x)+F(y),
\end{align*}
for $x,y\in\R^m$.
Consequently, there is a compact convex set $M$ in $\R^m$ such that $F=h_M$. The symmetry of $F$ implies that $M$ is 1-unconditional. Moreover, by \eqref{geneq7.4} again and (\ref{nn}), we have
\begin{eqnarray*}
h_{*(K_1,\ldots,K_m)}(x)&=&f\left(h_{K_1}(x),\ldots,h_{K_m}(x)\right)
=F\left(h_{K_1}(x),\ldots,h_{K_m}(x)\right)\\
&=&h_{M}\left(h_{K_1}(x),\ldots,h_{K_m}(x)\right)
=h_{\oplus_M}\left(h_{K_1}(x),\ldots,h_{K_m}(x)\right),
\end{eqnarray*}
for all $K_1,\ldots,K_m\in\K^n_s$ and all $x\in \R^n$.  This proves (\ref{geneq7.6}) and (\ref{geneq7.6b}), and the uniqueness of $M$ follows as in the proof of \cite[Lemma~7.6]{GHW}.
\end{proof}

\begin{cor}\label{contaff}
An operation $*:\left({\mathcal{K}}^n_s\right)^m\rightarrow {\mathcal{K}}^n$  is projection covariant if and only if it is continuous and $GL(n)$ covariant (and hence homogeneous of degree 1).
\end{cor}

\begin{proof}
If $*$ is continuous and $GL(n)$ covariant, then it is projection covariant by Lemma~\ref{CGimpliesPm} and homogeneous of degree 1.  Since $\oplus_M:\left({\mathcal{K}}^n_s\right)^m\rightarrow {\mathcal{K}}_s^n$ is continuous and $GL(n)$ covariant, the converse follows from Theorem~\ref{Thmnew7.6}.
\end{proof}

The following two results are obtained by a straightforward modification of the proofs of \cite[Theorem~9.7 and Corollary~9.9]{GHW}.

\begin{thm}\label{thm275b}
The $m$-ary operation $*:\left({\mathcal{K}}^n\right)^m\rightarrow {\mathcal{K}}^n$ (or $*:\left({\mathcal{K}}^n_o\right)^m\rightarrow {\mathcal{K}}^n$) is projection covariant if and only if there is a nonempty closed convex set $M$ in $\R^{2m}$ such that
\begin{equation}\label{nnnb}
h_{*(K_1,\ldots,K_m)}(x)=h_M\left(h_{K_1}(-x),h_{K_1}(x), \ldots,h_{K_m}(-x), h_{K_m}(x)\right),
\end{equation}
for all $K_1,\ldots,K_m\in {\mathcal{K}}^n$ (or $K_1,\ldots,K_m\in{\mathcal{K}}^n_o$, respectively) and $x\in\R^n$.
\end{thm}

\begin{cor}\label{contaff7}
An operation $*:\left({\mathcal{K}}^n\right)^m\rightarrow {\mathcal{K}}^n$ (or $*:\left({\mathcal{K}}^n_o\right)^m\rightarrow {\mathcal{K}}^n$) is projection covariant if and only if it is continuous and $GL(n)$ covariant (and hence homogeneous of degree 1).
\end{cor}

\section{Orlicz norms and Orlicz-Minkowski integrals}\label{Onorms}

Recall the definition of the class $\Phi_m$ from Section~\ref{prelim}. If $\varphi\in \Phi_m$, $Z$ is a set, and $\mu$ is a nonnegative measure in $Z$, we denote by $L_{\varphi}(Z,\mu)$ the set of all $\mu$-measurable functions $f:Z\to \R^m$ of the form $f=(f_1,\dots,f_m)$, where $f_j:Z\to \R$, such that
\begin{equation}\label{Orliczfinite}
\int_{Z}\varphi\left(\frac{|f|(z)}{\lambda}\right)\, d\mu(z)<\infty,
\end{equation}
for some $\lambda>0$,
with
$$|f|(z)=\left(|f_1(z)|,\dots,|f_m(z)|\right).$$
For $f\in L_{\varphi}(Z,\mu)$ we define
\begin{equation}\label{Orlnorm}
\| f\|_{\varphi}=\inf\left\{\lambda>0:\int_{Z} \varphi\left(\frac{|f|(z)}{\lambda}\right)\, d\mu(z)\le 1\right\}.
\end{equation}
Then $L_{\varphi}(Z,\mu)$ is a
vector space and $\|\cdot\|_{\varphi}$
is a norm, called an {\em Orlicz} (or {\em Luxemburg}) {\em norm} (see \cite[p.~49]{RaoRen91}). (Formally, $L_{\varphi}(Z,\mu)$ is considered here as a set of equivalence classes.)  The infimum in \eqref{Orlnorm}
is not taken over the empty set, since \eqref{Orliczfinite}, the dominated convergence theorem, and $\varphi(o)=0$ imply that
$\int_Z\varphi\left(|f|(z)/\lambda\right)\, d\mu(z)\le 1$ if $\lambda$ is sufficiently large. We also remark that $\varphi\in\Phi_m$ can be
extended to a map $\widetilde\varphi:\R^m\to[0,\infty)$ by defining $\widetilde\varphi(x)=\varphi(|x_1|,\ldots,|x_m|)$ for $x\in\R^m$. Then $\widetilde\varphi$ is convex on $\R^m$ (compare the arguments in the proof of Proposition \ref{triangleprp}) and $\varphi(|f|(z))=\widetilde\varphi(f(z))$, for $z\in Z$.

For the convenience of the reader, we provide a proof of the triangle inequality, since even when $m=1$, assumptions on $\varphi$ vary widely throughout the literature and some proofs (for example, that in \cite[p.~79]{KraR61}) are not directly applicable for our class $\Phi_m$.

\begin{prp}\label{triangleprp}
Let $\varphi\in {\Phi_m}$ and $f,g\in L_{\varphi}(Z,\mu)$. Then $f+g\in L_{\varphi}(Z,\mu)$ and
$$\|f+g\|_{\varphi}\le \|f\|_{\varphi}+\|g\|_{\varphi}.$$
\end{prp}

\begin{proof}
Let $\ee>0$. Then in view of (\ref{Orlnorm}), there are $\lambda_1,\lambda_2>0$ such that
\begin{equation}\label{lam12}
\int_{Z}\varphi\left(\frac{|f|(z)}{\lambda_1}
\right)\, d\mu(z)\le 1~~\quad~~{\text{and}}~~\quad~~
\int_{Z}\varphi\left(\frac{|g|(z)}{\lambda_2}\right)\, d\mu(z)\le 1,
\end{equation}
$\lambda_1<\|f\|_{\varphi}+\ee$, and $\lambda_2<\|g\|_{\varphi}+\ee$.  Using the fact that $\varphi$ is convex on $[0,\infty)^m$ and
increasing in each variable, $|f_j(z)+g_j(z)|\le |f_j(z)|+|g_j(z)|$ for $z\in Z$ and $j=1,\ldots,m$, and (\ref{lam12}), we obtain
\begin{eqnarray*}
\lefteqn{\int_{Z}\varphi\left(\frac{|f+g|(z)}{\lambda_1+\lambda_2}\right)\, d\mu(z)}\\
&\le & \int_{Z}\varphi\left(\frac{\lambda_1}{\lambda_1+\lambda_2}
\frac{|f|(z)}{\lambda_1}+\frac{\lambda_2}{\lambda_1+\lambda_2}
\frac{|g|(z)}{\lambda_2}\right)\,d\mu(z)\\
&\le &\frac{\lambda_1}{\lambda_1+\lambda_2}
\int_{Z}\varphi\left(\frac{|f|(z)}{\lambda_1}
\right)\, d\mu(z)+
\frac{\lambda_2}{\lambda_1+\lambda_2}
\int_{Z}\varphi\left(\frac{|g|(z)}{\lambda_2}\right)\, d\mu(z)\le 1.
\end{eqnarray*}
This shows that $f+g\in L_{\varphi}(Z,\mu)$ and
$$\|f+g\|_{\varphi}\le \lambda_1+\lambda_2\le  \|f\|_{\varphi}+\|g\|_{\varphi}+2\ee,$$
from which the result follows.
\end{proof}

For our application, we take $Z=({\K}^n_o)^m$. For each $x\in\R^n$, we define a continuous function $h_x:\left({\K}^n_o\right)^m\to [0,\infty)^m$ by letting
$$h_x(K_1,\dots,K_m)=\left(h_{K_1}(x),\dots,h_{K_m}(x)\right),
$$
for all $(K_1,\dots,K_m)\in \left({\K}^n_o\right)^m$.  In order that $h_x\in L_{\varphi}\left(\left({\K}^n_o\right)^m,\mu\right)$,
i.e., that (\ref{Orliczfinite}) holds when $f=h_x$, we require that the measure $\mu$ in $\left({\K}^n_o\right)^m$ is a Borel measure such
that, for each $x\in\R^n$,
\begin{equation}\label{finitecondn}
\int_{\left({\K}^n_o\right)^m}\varphi\left(\frac{h_{K_1}(x)}{\lambda},\dots,\frac{h_{K_m}(x)}{\lambda}\right)\, d\mu(K_1,\dots,K_m)<\infty,
\end{equation}
for some $\lambda>0$. (It is sufficient to require this for all $x\in S^{n-1}$.)
We refer to an integral of this type as an {\em Orlicz-Minkowski integral}.

Next, we define
\begin{eqnarray}
\label{newdeff}
h_{\Cpm}(x)&=&\|h_x\|_{\varphi}\nonumber\\
&=&\inf\left\{\lambda>0: \int_{\left({\K}^n_o\right)^m}\varphi\left(\frac{h_{K_1}(x)}{\lambda}
,\dots,\frac{h_{K_m}(x)}{\lambda}\right)\, d\mu(K_1,\dots,K_m)\le 1\right\},
\end{eqnarray}
for all $x\in \R^n$.

\begin{lem}\label{OrlMinkInt}
The function $h_{\Cpm}$ defined by \eqref{newdeff} is the support function of a compact convex set $\Cpm$ in $\R^n$.
\end{lem}

\begin{proof}
From the right-hand side of (\ref{newdeff}) and the fact that support functions are homogeneous of degree 1, it is easy to see that $h_{\Cpm}(rx)=
rh_{\Cpm}(x)$, for $r\ge 0$ and $x\in \R^n$.

Using (\ref{newdeff}), the subadditivity of support functions, the fact that $\varphi$ is increasing in each variable, and Proposition~\ref{triangleprp}, we obtain
\begin{align*}
\lefteqn{h_{\Cpm}(x+y)}\\
&=\inf\left\{\lambda>0:\int_{\left({\K}^n_o\right)^m}\varphi
\left(\frac{h_{K_1}(x+y)}{\lambda}
,\dots,\frac{h_{K_m}(x+y)}{\lambda}\right)\, d\mu(K_1,\dots,K_m)\le 1\right\}\\
&\le\inf\left\{\lambda>0:\int_{\left({\K}^n_o\right)^m}\varphi
\left(\frac{h_{K_1}(x)+h_{K_1}(y)}{\lambda}
,\dots,\frac{h_{K_m}(x)+h_{K_m}(y)}{\lambda}\right)\, d\mu(K_1,\dots,K_m)\le 1\right\}\\
&=\|h_x+h_y\|_{\varphi}\le \|h_x\|_{\varphi} + \|h_y\|_{\varphi}=h_{\Cpm}(x)+h_{\Cpm}(y),
\end{align*}
for all $x,y\in \R^n$.  This completes the proof.
\end{proof}

In the special case when $m=1$, we conclude that
\begin{equation}
\label{spnewdeff}
h_{\Cpm}(x)
=\inf\left\{\lambda>0: \int_{{\K}^n_o}
{\varphi}\left(\frac{h_K(x)}{\lambda}\right)\, d\mu(K)\le 1\right\},
\end{equation}
for all $x\in \R^n$, defines the support function of a compact convex set $\Cpm$ in $\R^n$.

Before continuing, we present some important special cases of this construction when $m=1$.

\begin{ex}\label{measurex}
{\em
${\mathrm{(i)}}$ Let $K\in\K^n_{oo}$. In \cite{LYZ7}, the {\em Orlicz projection body} $\Pi_{\varphi}K$ of $K$ is defined by
\begin{equation}\label{Orlprojsupp}
h_{\Pi_{\varphi}K}(u)=\inf\left\{\lambda>0:\int_{S^{n-1}}\varphi
\left(\frac{|u\cdot v|}{\lambda h_K(v)}\right)\,d\overline{V}_n(K,v)\le 1\right\},
\end{equation}
for $u\in S^{n-1}$, where $\overline{V}_n(K,\cdot)$ is the normalized cone measure for $K$ defined by (\ref{normconem}).
Define a measure $\pi_K$ in $\K^n_{o}$, concentrated on the set of $o$-symmetric line segments in $\R^n$, by
$$
\pi_K({\mathcal{L}})=\int_{S^{n-1}}1_{\left\{\left[-\frac{v}{h_K(v)},
\frac{v}{h_K(v)}\right]\in\, {\mathcal{L}}\right\}}\,d\overline{V}_n(K,v),
$$
for  Borel sets ${\mathcal{L}}\subset \K^n_{o}$. Then for all $x\in \R^n$, we have
$$
\int_{\K^n_o}
\varphi\left(\frac{h_L(x)}{\lambda}\right)\, d\pi_K(L)=\int_{S^{n-1}}
\varphi\left(\frac{|v\cdot x|}{\lambda h_K(v)}\right)\, d\overline{V}_n(K,v).
$$
Comparing (\ref{spnewdeff}) and (\ref{Orlprojsupp}), we see that
$$
\Cppi=\Pi_\varphi K.
$$
More generally, we can obtain a natural definition of Orlicz zonoids as a special case of (\ref{spnewdeff}).  Also, the {\em asymmetric Orlicz projection body} $\Pi_\varphi^+ K=\Cppip$ of $K$ (the Orlicz analog of the asymmetric $L_p$ projection body introduced by Ludwig \cite{Lud}) can be obtained by defining a measure $\pi_K^+$, concentrated on the set of line segments with one endpoint at the origin in $\R^n$, by
$$
\pi_K^+({\mathcal{L}})=\int_{S^{n-1}}1_{\left\{\left[o,
\frac{v}{h_K(v)}\right]\in\, {\mathcal{L}}\right\}}\,d\overline{V}_n(K,v),
$$
again for Borel sets ${\mathcal{L}}\subset \K^n_{o}$.

(ii) Let $K\in \K^n_{oo}$.  The {\em Orlicz centroid body} $\Gamma_{\varphi}K$ of $K$ is defined in \cite{LYZ8} (actually for any star body $K$) by
$$
h_{\Gamma_{\varphi}K}(u)=\inf\left\{\lambda>0:\frac{1}{V(K)}
\int_{K}\varphi\left(\frac{|u\cdot x|}{\lambda}\right)\,dx\le 1\right\},
$$
for $u\in S^{n-1}$.  Define a measure $\gamma_K$ in $\K^n_{o}$, concentrated on the set of $o$-symmetric line segments contained in $K$, by
$$
\gamma_K({\mathcal{L}})=\frac{1}{V(K)}\int_{K}
1_{\{[-x,x]\in\,{\mathcal{L}}\}}\,dx,
$$
for Borel sets ${\mathcal{L}}\subset \K^n_{o}$.  Then
$$
\Cpg=\Gamma_\varphi K.
$$
An asymmetric Orlicz centroid body $\Gamma_\varphi^+ K=\Cpgp$ can also be obtained by defining the measure $\gamma_K^+$ in the obvious way.}
\end{ex}

We now record a few properties of the construction that will be useful in the next section. For a linear map $A:\R^n\to\R^n$ and $K_1,\ldots,K_m\in\K^n_o$ we define a continuous (and hence Borel measurable) map $A:(\K^n_o)^m\to(\K^n_o)^m$ by
$$
A(K_1,\ldots,K_m)=(AK_1,\ldots,A K_m).
$$
If $\mu$ is a Borel measure in $(\K^n_o)^m$, then $A\mu$ denotes the image measure of $\mu$ under $A$, i.e.,
$$
(A\mu)(\mathcal{E})=\mu(A^{-1}\mathcal{E})=\int_{(\K^n_o)^m} 1_{\mathcal{E}}(AK_1,\ldots,AK_m)\, d\mu(K_1,\ldots,K_m),
$$
where $\mathcal{E}\subset (\K^n_o)^m$ is a Borel set.

\begin{lem}\label{constlem}
Let $\varphi\in \Phi_m$.

\smallskip

\noindent ${\mathrm{(i)}}$ Let $\mu$ and $\mu_i$, $i\in \N$, be finite measures in $\left({\K}^n_o\right)^m$ such that the supports of $\mu$ and $\mu_i$, $i\in\N$, are contained in a common compact subset of $\left({\K}^n_o\right)^m$, and $\mu_i\to \mu$ weakly as $i\to\infty$.  Then $\Cpmi\to \Cpm$ as $i\to\infty$ in the Hausdorff metric.

\smallskip

\noindent ${\mathrm{(ii)}}$  Let $\mu$ be a measure in $\left({\K}^n_o\right)^m$ such that \eqref{finitecondn} holds.
If $A:\R^n\to\R^n$ is a linear map, then
$$
A\left(\Cpm\right)=\CpAm.
$$
\end{lem}

\begin{proof}
(i) Let $\lambda>0$ and $x\in\R^n$. Then the map $(K_1,\ldots,K_m)\mapsto\varphi
\left(h_{K_1}(x)/\lambda,\ldots,h_{K_m}(x)/\lambda\right)$
is continuous on $\left(\mathcal{K}^n_o\right)^m$ and therefore bounded on a compact set containing the supports of $\mu$ and $\mu_i$, $i\in\N$. Since $\mu$ and $\mu_i$, $i\in\N$, are finite, these measures satisfy \eqref{finitecondn} for all $\lambda>0$ and $x\in\R^n$. Hence, for each $\lambda>0$ and $x\in\R^n$, we have
\begin{eqnarray*}
\lefteqn{\int_{\left({\K}^n_o\right)^m}\varphi\left(\frac{h_{K_1}(x)}{\lambda},
\dots,\frac{h_{K_m}(x)}{\lambda}\right)\, d\mu_i(K_1,\dots,K_m)}\\
&\to & \int_{\left({\K}^n_o\right)^m}\varphi\left(\frac{h_{K_1}(x)}{\lambda},\dots,
\frac{h_{K_m}(x)}{\lambda}\right)\, d\mu(K_1,\dots,K_m),
\end{eqnarray*}
as $i\to\infty$. It then follows from (\ref{newdeff}) that $h_{\Cpmi}(x)\to h_{\Cpm}(x)$ as $i\to\infty$.  Since pointwise convergence of support functions implies convergence in the Hausdorff metric (see \cite[Theorem~1.8.12]{Sch93}), we are done.

(ii) If $x\in \R^n$, then by (\ref{suppchange}) and (\ref{newdeff}),
\begin{align*}
h_{A\left(\Cpm\right)}(x)
&=h_{\Cpm}(A^tx)\\
&=\inf\left\{\lambda>0: \int_{\left({\K}^n_o\right)^m}\varphi\left(\frac{h_{K_1}(A^tx)}{\lambda}
,\dots,\frac{h_{K_m}(A^tx)}{\lambda}\right)\, d\mu(K_1,\dots,K_m)\le 1\right\}\\
&=\inf\left\{\lambda>0: \int_{\left({\K}^n_o\right)^m}\varphi\left(\frac{h_{AK_1}(x)}{\lambda}
,\dots,\frac{h_{AK_m}(x)}{\lambda}\right)\, d\mu(K_1,\dots,K_m)\le 1\right\}\\
&= \inf\left\{\lambda>0: \int_{\left({\K}^n_o\right)^m}\varphi\left(\frac{h_{K_1}(x)}{\lambda}
,\dots,\frac{h_{K_m}(x)}{\lambda}\right)\, d(A\mu)(K_1,\dots,K_m)\le 1\right\}\\
&=h_{\CpAm}(x).
\end{align*}
\end{proof}

\section{Orlicz addition}\label{Orlicz}

Recall the definition of $\Phi_m$ from Section~\ref{prelim}.  Let $m\ge 2$, let $\varphi\in \Phi_m$, and for $j=1,\dots,m$, let $K_j\in {{\K}}^n_o$. Define a measure $\mu$ in $\left({{\K}}^n_o\right)^m$ by
$$
\mu=\delta_{K_1}\times\cdots\times
\delta_{K_m}.
$$
The corresponding {\em Orlicz sum} of $K_1,\dots,K_m$ is defined to be $\Cpm$, where $\Cpm$ is as in (\ref{newdeff}), and is denoted by $+_{\varphi}(K_1,\dots,K_m)$.  This means that
\begin{equation}\label{OrlComb}
h_{+_{\varphi}(K_1,\dots,K_m)}(x)=\inf\left\{\lambda>0: \varphi\left(\frac{h_{K_1}(x)}{\lambda}
,\dots,\frac{h_{K_m}(x)}{\lambda}\right)\le 1\right\},
\end{equation}
for all $x\in \R^n$.

Equivalently, the Orlicz sum $+_{\varphi}(K_1,\dots,K_m)$ can be defined implicitly (and uniquely) by
\begin{equation}\label{Orldef}
\varphi\left( \frac{h_{K_1}(x)}{h_{+_{\varphi}(K_1,\dots,K_m)}(x)},\dots, \frac{h_{K_m}(x)}{h_{+_{\varphi}(K_1,\dots,K_m)}(x)}\right) = 1,
\end{equation}
if $h_{K_1}(x)+\cdots+h_{K_m}(x)>0$ and by $h_{+_{\varphi}(K_1,\dots,K_m)}(x)=0$ if $h_{K_1}(x)=\cdots=h_{K_m}(x)=0$, for all $x\in \R^n$. Note that $h_{+_{\varphi}(K_1,\dots,K_m)}(x)=0$ implies that $h_{K_1}(x)+\cdots+h_{K_m}(x)=0$ and hence that $h_{K_1}(x)=\cdots=h_{K_m}(x)=0$. Also, if $h_{K_j}(x)=0$ for all $j\neq j_0$, then \eqref{Orldef} and our assumptions on $\varphi$ yield $h_{+_{\varphi}(K_1,\dots,K_m)}(x)=h_{K_{j_0}}(x)$.

An important special case is obtained when
\begin{equation}\label{spphi}
\varphi(x_1,\dots,x_m)=\sum_{j=1}^m\varphi_j(x_j),
\end{equation}
for some fixed $\varphi_j\in \Phi=\Phi_1$, $j=1,\dots,m$, such that $\varphi_1(1)=\cdots=\varphi_m(1)=1$.  We then write
$+_{\varphi}(K_1,\dots,K_m)=K_1+_{\varphi}\cdots+_{\varphi}K_m$.  This means that $K_1+_{\varphi}\cdots+_{\varphi}K_m$ is defined either by
$$
h_{K_1+_{\varphi}\cdots+_{\varphi}K_m}(x)=
\inf\left\{\lambda>0:\sum_{j=1}^m\varphi_j\left(\frac{h_{K_j}(x)}
{\lambda}\right)\le 1\right\},
$$
for all $x\in \R^n$, or by the corresponding special case of (\ref{Orldef}).

\begin{rem}\label{Remark4.2}
{\em Suppose that $\varphi(x_1,x_2)=\varphi_1(x_1)+\varphi_2(x_2)$, where $\varphi_1,\varphi_2\in \Phi$ are such that $\tau=\tau(\varphi_2)>0$ (see (\ref{tau})).  Suppose also that $K,L\in{\mathcal K}^n_{o}$ satisfy $h_L(x)\le \tau h_K(x)$, for some $x\neq o$.  Then $h_{K+_{\varphi} L}(x)=h_K(x)$, since
$$
\varphi_1\left(\frac{h_K(x)}{h_K(x)}\right)+\varphi_2\left(\frac{h_L(x)}{h_K(x)}
\right)=\varphi_1(1)+0=1,
$$
if $h_K(x)>0$ (compare (\ref{Orldef})).  In particular, if $K\in\K^n_{oo}$ and if $L\in\K^n_o$ is contained in a sufficiently small ball, then $K+_{\varphi} L=K$.}
\end{rem}

The following theorem establishes some properties (defined in Section~\ref{Maddition}) of Orlicz addition.

\begin{thm}\label{Orthm1}
If $\varphi\in \Phi_m$, then  Orlicz addition $+_{\varphi}:\left({\mathcal K}^n_{o}\right)^m\rightarrow {\mathcal K}^n_{o}$ is monotonic, continuous,  $GL(n)$ covariant, projection covariant, has the identity property, and $+_{\varphi}:\left({\mathcal{K}}^n_s\right)^m\rightarrow {\mathcal{K}}^n_s$.
\end{thm}

\begin{proof}
We first claim that $+_{\varphi}$ is monotonic.  To see this, let $K_j\subset L_j$, where $K_j,L_j\in {\mathcal K}^n_{o}$, $j=1,\dots,m$. Let $x\in\R^n$. If $h_{K_1}(x)=\cdots=h_{K_m}(x)=0$, then $h_{+_{\varphi}(K_1,\dots,K_m)}(x)=0\le h_{+_{\varphi}(L_1,K_2,\dots,K_m)}(x)$. If $h_{K_1}(x)+\cdots+h_{K_m}(x)>0$, then $h_{L_1}(x)+h_{K_2}(x)+\cdots+h_{K_m}(x)>0$ and using (\ref{Orldef}), $K_1\subset L_1$, and the fact that $\varphi$ is increasing in the first variable, we obtain
\begin{eqnarray*}
\lefteqn{ \varphi\left( \frac{h_{L_1}(x)}{h_{+_{\varphi}(L_1,K_2,\dots,K_m)}(x)},
\frac{h_{K_2}(x)}{h_{+_{\varphi}(L_1,K_2,\dots,K_m)}(x)},
\dots,\frac{h_{K_m}(x)}{h_{+_{\varphi}(L_1,K_2,\dots,K_m)}(x)}\right)
}\\
&=&1=\varphi\left( \frac{h_{K_1}(x)}{h_{+_{\varphi}(K_1,K_2,\dots,K_m)}(x)},
\frac{h_{K_2}(x)}{h_{+_{\varphi}(K_1,K_2,\dots,K_m)}(x)},
\dots,\frac{h_{K_m}(x)}{h_{+_{\varphi}(K_1,K_2,\dots,K_m)}(x)}\right)\\
&\le&  \varphi\left( \frac{h_{L_1}(x)}{h_{+_{\varphi}(K_1,K_2,\dots,K_m)}(x)},
\frac{h_{K_2}(x)}{h_{+_{\varphi}(K_1,K_2,\dots,K_m)}(x)},
\dots,\frac{h_{K_m}(x)}{h_{+_{\varphi}(K_1,K_2,\dots,K_m)}(x)}\right),
\end{eqnarray*}
which again implies that $h_{+_{\varphi}(K_1,\dots,K_m)}(x)\le h_{+_{\varphi}(L_1,K_2,\dots,K_m)}(x)$. By repeating this argument for each of the other $m-1$ variables, we obtain $h_{+_{\varphi}(K_1,\dots,K_m)}(x)\le h_{+_{\varphi}(L_1,\dots,L_m)}(x)$.  This proves the claim.

Next, we claim that $+_{\varphi}$ is continuous.  Indeed, let $K_{ij}\in {\mathcal K}^n_{o}$, $i\in \N\cup\{0\}$, $j=1,\dots,m$, be such that $K_{ij}\to K_{0j}$ as $i\to \infty$, and let $x\in\R^n$.  If
$\mu_{i}=\delta_{K_{i1}}\times\cdots\times
\delta_{K_{im}}$,
for $i\in \N\cup\{0\}$, and $f:\left({\mathcal K}^n_{o}\right)^m\to\R$ is continuous, then
\begin{eqnarray*}
\int_{\left({\mathcal K}^n_{o}\right)^m}f(L_1,\dots,L_m)\,d\mu_{i}(L_1,\dots,L_m)&=&
f(K_{i1},\dots,K_{im})\to
f(K_{01},\dots,K_{0m})\\
&=&\int_{\left({\mathcal K}^n_{o}\right)^m}f(L_1,\dots,L_m)\,d\mu_{0}(L_1,\dots,L_m),
\end{eqnarray*}
as $i\to \infty$.  It follows that $\mu_i\to \mu_0$ weakly.  By Lemma~\ref{constlem}(i), $+_{\varphi}(K_{i1},\dots,K_{im})=\Cpmi\to \Cpmo=+_{\varphi}(K_{01},\dots,K_{0m})$ as $i\to\infty$ in the Hausdorff metric, as required.

Let $K_j\in {\mathcal K}^n_{o}$, $j=1,\dots,m$, and let $A:\R^n\to\R^n$ be linear. If $\mu=\delta_{K_{1}}\times\cdots\times
\delta_{K_{m}}$, then by Lemma~\ref{constlem}(ii), we have
$$A(+_{\varphi}(K_{1},\dots,K_{m}))=A\left(\Cpm\right)
=\CpAm=+_{\varphi}(AK_{1},\dots,AK_{m}),$$
since  $A\mu= A\left(\delta_{K_{1}}\times\cdots\times
\delta_{K_{m}}\right)=\delta_{A K_{1}}\times\cdots\times
\delta_{A K_{m}}$.  In particular, it follows that $+_\varphi$ is $GL(n)$ covariant and projection covariant. Of course, the latter is also a consequence of the continuity and $GL(n)$ covariance of $+_\varphi$, by Lemma~\ref{CGimpliesPm}.

The identity property is obvious from (\ref{Orldef}) and the remarks thereafter.

The $GL(n)$ covariance of $+_{\varphi}$, applied with the transformation $Ax=-x$, $x\in \R^n$, shows that $+_{\varphi}:\left({\mathcal{K}}^n_s\right)^m\rightarrow {\mathcal{K}}^n_s$.
\end{proof}

The next theorem shows that Orlicz sums can also be considered as $M$-sums.

\begin{thm}\label{Orthm2}
If $\varphi\in \Phi_m$, a 1-unconditional convex body $J_{\varphi}$ in $[-1,1]^m$, containing $e_1,\ldots,e_m$ in its boundary, is defined by
\begin{equation}\label{Jphi}
J_{\varphi}\cap [0,\infty)^m=\{(x_1,\ldots,x_m)\in [0,1]^m: \varphi(x_1,\ldots,x_m)\le 1\}.
\end{equation}
Then $+_{\varphi}:\left({\mathcal{K}}^n_s\right)^m\rightarrow {\mathcal{K}}^n_s$ (or $+_{\varphi}:\left({\mathcal{K}}^n_o\right)^m\rightarrow {\mathcal{K}}^n_o$) is $M$-addition $\oplus_M$ with $M=J_{\varphi}^{\circ}$ (or $M=J_{\varphi}^{\circ}\cap [0,\infty)^m$, respectively). Moreover, the formula
\begin{equation}\label{nnnn}
h_{+_\varphi (K_1,\ldots,K_m)}(x)=h_{J_{\varphi}^{\circ}}\left(h_{K_1}(x),\ldots,h_{K_m}(x)\right),
\end{equation}
holds for all $K_1,\ldots,K_m\in {\mathcal{K}}^n_o$ and $x\in \R^n$.
\end{thm}

\begin{proof}
We claim that the set $C$ defined by the right-hand side of (\ref{Jphi}) is a convex body contained in $[0,1]^m$ and containing $e_1,\ldots,e_m$ in its boundary.  To see this, note that $C$ is convex since it is a sublevel set of a convex function. Also, $\varphi(e_j)=1$ for $j=1,\ldots,m$, since $\varphi\in \Phi_m$, so $e_1,\ldots,e_m\in C$. Then the fact that $\varphi$ is increasing in each variable ensures that $C\subset [0,1]^m$ and that $(\alpha_1x_1,\ldots,\alpha_mx_m)\in C$ whenever
$(x_1,\ldots,x_m)\in C$ and $\alpha_1,\ldots,\alpha_m\in[0,1]$.  This proves the claim.  Moreover, it follows that
$$J_{\varphi}=\{(\alpha_1x_1,\ldots,\alpha_mx_m): (x_1,\ldots,x_m)\in C, |\alpha_j|\le 1, j=1,\ldots,m\}$$
is a 1-unconditional convex body such that $J_{\varphi}\cap [0,\infty)^m=C$.  Then $J_{\varphi}^{\circ}$ is also a 1-unconditional convex body.

Next, we prove that $+_{\varphi}:\left({\mathcal{K}}^n_s\right)^m\rightarrow {\mathcal{K}}^n_s$ is $M$-addition with $M=J_{\varphi}^\circ$. Since $+_{\varphi}$ is projection covariant by Theorem~\ref{Orthm1}, it is $M$-addition for some 1-unconditional compact convex subset $M$ of $\R^m$, by Theorem~\ref{Thmnew7.6}. We have to show that $M=J_{\varphi}^{\circ}$.  To see this, note first that by (\ref{nn}) we have
\begin{equation}\label{ngood}
h_{+_{\varphi} (K_1,\ldots,K_m)}(x)=h_M\left(h_{K_1}(x),\ldots,h_{K_m}(x)\right),
\end{equation}
for all $K_1,\ldots,K_m\in {\mathcal{K}}^n_s$ and $x\in \R^n$. Suppose that $w=(w_1,\dots,w_m)\in(0,\infty)^m$ and let $K_j=[-w_je_1,w_je_1]\in\K^n_s$, for $j=1,\ldots,m$,  and $x=e_1$.  Then by (\ref{Orldef}) and \eqref{ngood}, we have
\begin{equation}\label{eq2}
\varphi\left( \frac{w}{h_M(w)}\right) = 1.
\end{equation}
Thus $h_M(w)$ is the number such that $w/h_M(w)$ belongs to the boundary of ${J_{\varphi}}$. But by the definition (\ref{radialfunction}) of the radial function, this means that
$$h_M(w)=\frac{1}{\rho_{J_{\varphi}}(w)}=h_{J_{\varphi}^{\circ}}(w).$$
(The latter equation holds by \cite[(0.36), p.~20]{Gar06}, for example, because $J_{\varphi}$ is a convex body with $o\in \inte J_{\varphi}$.) Therefore $h_M=h_{J_{\varphi}^{\circ}}$ on $(0,\infty)^m$, and by continuity also on $[0,\infty)^m$.  Since $J_{\varphi}^{\circ}$ and $M$ are 1-unconditional, this proves that $M=J_{\varphi}^{\circ}$.

Now suppose that $+_{\varphi}:\left({\mathcal{K}}^n_o\right)^m\rightarrow {\mathcal{K}}^n_o$.  The projection covariance of $+_{\varphi}$ follows as before from Theorem~\ref{Orthm1}, so by (\ref{nnnb}), we have
$$
h_{+_\varphi(K_1,\ldots,K_m)}(x)=h_{M}\left(h_{K_1}(-x),h_{K_1}(x),\ldots,h_{K_m}(-x),h_{K_m}(x)\right),
$$
for some nonempty closed convex set $M$ in $\R^{2m}$ and all $K_1,\ldots,K_m\in {\mathcal{K}}^n_o$ and $x\in\R^n$. Let $x\in \R^n$ and $K_1,\ldots,K_m\in {\mathcal{K}}^n_o$.  Choose $K'_1,\ldots,K_m'\in {\mathcal{K}}^n_o$ such that $h_{K'_1}(x)=h_{K_1}(x),\ldots, h_{K_m'}(x)=h_{K_m}(x)$, and $h_{K'_1}(-x)=\cdots=h_{K_m'}(-x)=0$.  By (\ref{Orldef}), we have
$$
h_{+_\varphi(K_1,\ldots,K_m)}(x) = h_{+_\varphi(K_1',\ldots,K_m')}(x)
$$
and hence
\begin{align*}
h_M(h_{K_1}(-x),h_{K_1}(x),\ldots,h_{K_m}(-x),h_{K_m}(x))&= h_M(h_{K'_1}(-x),h_{K'_1}(x),\ldots,h_{K_m'}(-x),h_{K_m'}(x))\\
&= h_M(0,h_{K_1}(x),\ldots,0,h_{K_m}(x)).
\end{align*}
Replacing $M$ by its projection on the $\{x_2,x_4,\ldots,x_{2m}\}$-plane in $\R^{2m}$, using (\ref{hproj}), and identifying the latter with $\R^m$, we see that (\ref{ngood}) holds for $K_1,\ldots,K_m\in \K^n_{o}$.  Following the proof above for the $o$-symmetric case, we obtain
$h_M =h_{J_{\varphi}^{\circ}} $ on $(0,\infty)^m$.  Since support functions are continuous and $h_{K_1}(x),\ldots,h_{K_m}(x)\ge 0$ for all $K_1,\ldots,K_m\in {\mathcal{K}}^n_o$ and $x\in\R^n$, (\ref{nnnn}) follows.

Observe that (\ref{nnnn}) is unaffected if $J_{\varphi}^{\circ}$ is replaced by $J_{\varphi}^{\circ}\cap [0,\infty)^m$, since $J_{\varphi}^{\circ}$ is 1-unconditional.  By \cite[Theorem~6.1(ii)]{GHW}, when
$M=J_{\varphi}^{\circ}\cap [0,\infty)^m$, $\oplus_M$ maps $\left({\mathcal{K}}^n_o\right)^2$ to ${\mathcal{K}}^n_o$, and by \cite[Theorem~6.5(i)]{GHW}, we have
$$
h_{\oplus_{J_{\varphi}^{\circ}\cap [0,\infty)^m}(K_1,\ldots,K_m)}(x)=
h_{J_{\varphi}^{\circ}\cap [0,\infty)^m}\left(h_{K_1}(x),\ldots,h_{K_m}(x)\right),$$
for all $K_1,\ldots,K_m\in {\mathcal{K}}^n_o$ and $x\in\R^n$.
It follows that $+_{\varphi}:\left({\mathcal{K}}^n_o\right)^m\rightarrow {\mathcal{K}}^n_o$ is $M$-addition with $M=J_{\varphi}^{\circ}\cap [0,\infty)^m$.
\end{proof}

Clearly, we have $\text{conv}\{\pm e_1,\ldots,\pm e_m\}\subset J_{\varphi}\subset
[-1,1]^m$.  Theorem~\ref{Orthm2} raises the question as to which convex bodies have a boundary representation of the form (\ref{Jphi}) and which special forms it might take.  The next two results address this problem.

\begin{thm}\label{inverseTh4.4general}
Let $K$ be a $1$-unconditional convex body in $\R^m$ that contains $e_1,\ldots,e_m$ in its boundary. Then there
is a homogeneous of degree 1 function $\varphi\in\Phi_m$ such that
\begin{equation}\label{Th4.4neweq}
\partial K\cap [0,1]^m=\left\{x\in [0,1]^m:
\varphi(x)= 1\right\}.
\end{equation}
\end{thm}

\begin{proof}
Let $K$ be as in the statement of the theorem. Then $o\in \inte K$ and
$$
h_{K^\circ}(x)=\frac{1}{\rho_K(x)}=\min\{\lambda\ge 0:x\in\lambda K\},
$$
for $x\in\R^n$, is the gauge function of $K$. Let $\varphi$ denote the restriction of $h_{K^\circ}$ to $[0,\infty)^m$. Since $h_{K^\circ}$ is convex and homogeneous of degree 1, $\varphi$ has the same properties. Moreover, since $K\subset[-1,1]^m$, $K=\{x\in\R^m:h_{K^\circ}(x)\le 1\}$, and $\partial K=\{x\in\R^m:h_{K^\circ}(x)= 1\}$, \eqref{Th4.4neweq} holds and $\varphi(e_j)=1$ for $j=1,\ldots,m$. Finally, we show that $\varphi$ is increasing in each variable and hence $\varphi\in \Phi_m$. Let $x=(x_1,\dots,x_m)\in [0,\infty)^m$ and $y=(y_1,\dots,y_m)\in [0,\infty)^m$ be such that $x_j\le y_j$ for $j=1,\ldots,m$. Let $\lambda\ge 0$ be such that $y\in\lambda \partial K$, i.e., $h_{K^\circ}(y)=\lambda$. Since $K$ is 1-unconditional, we have $x\in \sum_{j=1}^m[o,y_je_j]\subset \lambda K$, and therefore $\varphi(x)=h_{K^\circ}(x)\le \lambda=h_{K^\circ}(y)=\varphi(y)$.
\end{proof}

When $m=2$, the next result supplies special forms for $\varphi$, which show that with a single exception, \eqref{spphi} represents the general situation.

\begin{thm}\label{inverseTh4.4}
Let $K$ be a $1$-unconditional convex body in $\R^2$ that contains $e_1$ and $e_2$ in its boundary. If $K=[-1,1]^2$, then $\partial K\cap [0,1]^2$ is given by the equation
\begin{equation}\label{phimax}
\varphi(x_1,x_2)=\max\{x_1,x_2\}=1,
\end{equation}
for $x_1,x_2\ge 0$.  Otherwise, there are $\varphi_1, \varphi_2\in\Phi$ such that $\partial K\cap [0,1]^2$ is given by the equation
\begin{equation}\label{Jphi2}
\varphi(x_1,x_2)=\varphi_1(x_1)+\varphi_2(x_2)=1,
\end{equation}
for $x_1,x_2\ge 0$.  If $K\neq [-1,1]^2$ is also symmetric with respect to $x_1=x_2$, then there is a $\varphi_0\in \Phi$ such that \eqref{Jphi2} holds with $\varphi_1=\varphi_2=\varphi_0$.
\end{thm}

\begin{proof}
The case when $K=[-1,1]^2$ is clear, so we may assume that $K\neq [-1,1]^2$.  Then there are unique maximal $\tau_1,\tau_2\in [0,1)$ such that $({\tau_1},1)\in \partial K$ and $(1,{\tau_2})\in \partial K$.  (This notation will turn out to be consistent with (\ref{tau}).) There is a unique concave function $f:[0,1]\to[0,1]$ such that
$$
K\cap [0,1]^2=\{(t,\lambda f(t)):\lambda,t\in [0,1]\}.
$$
In particular, $f(t)=1$ for $t\in [0,{\tau_1}]$, $f(1)={\tau_2}$, and $f$ is strictly decreasing on $[{\tau_1},1]$, so $f:[{\tau_1},1]\to [{\tau_2},1]$ is a bijection with $f({\tau_1})=1$ and $f(1)={\tau_2}$. We extend $f^{-1}$ to a map $f^{-1}:[0,1]\to[0,1]$ by defining $f^{-1}(t)=1$ for $t\in [0,{\tau_2}]$.

We distinguish three cases. Firstly, if $f'_r({\tau_1})<0$, then $(f^{-1})'_l(1)>-\infty$ and we define
$$
\varphi_1(s)=\frac{\max\{s-{\tau_1},0\}}{1-{\tau_1}},$$
for all $s\ge 0$, and
$$
\varphi_2(s)=\begin{cases}
\frac{1-f^{-1}(s)}{1-{\tau_1}},& {\text{if $0\le s\le 1$,}}\\[1ex]
\frac{-(f^{-1})'_l(1)}{1-{\tau_1}}(s-1)+1,& {\text{if}}~~s\ge 1.
\end{cases}
$$
Then $\varphi_1(0)=\varphi_2(0)=0$ and $\varphi_1(1)=\varphi_2(1)=1$. Since $f^{-1}$ is concave, it is easy to check that $\varphi_2$ is convex. If $s\in [0,{\tau_1}]$, then
$\varphi_1(s)+\varphi_2(f(s))=0+\varphi_2(1)=1$. If $s\in ({\tau_1},1)$, then $f(s)\in({\tau_2},1)$, so $f^{-1}(f(s))=s$ and
$$
\varphi_1(s)+\varphi_2(f(s))=\frac{s-{\tau_1}}{1-{\tau_1}}+\frac{1-f^{-1}(f(s))}{1-{\tau_1}}
=\frac{s-{\tau_1}+1-s}{1-{\tau_1}}=1.
$$
If $s=1$, then $f(1)={\tau_2}$, so $\varphi_2({\tau_2})=0$ and hence again
$\varphi_1(1)+\varphi_2(f(1))=1+0=1$.  This settles the first case.

Secondly, if $f'_l(1)>-\infty$, then we define
\begin{align*}
\varphi_1(s)&=\begin{cases} \frac{1-f(s)}{1-\tau_2},&{\text{if $0\le s\le 1$,}}\\[1ex]
\frac{-f'_l(1)}{1-\tau_2}(s-1)+1,&{\text{if}}~~s\ge 1,
\end{cases}\\
\end{align*}
and
$$\varphi_2(s)=\frac{\max\{s-\tau_2,0\}}{1-\tau_2},$$
for all $s\ge 0$.  This case is completely symmetric to the first one and can therefore be settled by the same argument.

Thirdly, we assume that $f_r'({\tau_1})=0$ and $f_l'(1)=-\infty$. Then there is a unique $a\in ({\tau_1},1)$ such that $f_l'(a)\ge -1$ and $f_r'(a)\le -1$. Let $b=f(a)\in ({\tau_2},1)$, so that $f^{-1}(b)=a$, and define
\begin{align*}
\varphi_1(s)&=\begin{cases} \frac{1-f(s)}{2-a-b},&{\text{if $0\le s\le a$,}}\\[1ex]
\frac{s+1-a-b}{2-a-b},&{\text{if}}~~s\ge a,\\[1ex]
\end{cases}\\
\varphi_2(s)&=\begin{cases} \frac{1-f^{-1}(s)}{2-a-b},&{\text{if $0\le s\le b$,}}\\[1ex]
\frac{s+1-a-b}{2-a-b},&{\text{if}}~~s\ge b.\\[1ex]
\end{cases}
\end{align*}
It is easy to check that $\varphi_1$ and $\varphi_2$ are well defined, $\varphi_1(0)=\varphi_2(0)=0$, and $\varphi_1(1)=\varphi_2(1)=1$. Moreover, $\varphi_1$ is convex, since it is convex on $[0,a]$ and on $[a,\infty)$ and since
$f_l'(a)\ge -1$  implies that
$$
(\varphi_1)'_l(a)=-\frac{f_l'(a)}{2-a-b}\le \frac{1}{2-a-b}=(\varphi_1)'_r(a).
$$
Furthermore, $f_r'(a)\le -1$ implies that
$$
(f^{-1})'_l(b)=\frac{1}{f_r'(a)}\ge -1,
$$
so a similar argument shows that $\varphi_2$ is convex as well.

If $s\in [0,{\tau_1}]$, then $f(s)=1$ and $\varphi_1(s)+\varphi_2(f(s))=0+\varphi_2(1)=1$.
If $s\in [{\tau_1},a]$, then $f(s)\in [b,1]$ and
$$
\varphi_1(s)+\varphi_2(f(s))=\frac{1-f(s)}{2-a-b}+\frac{f(s)+1-a-b}{2-a-b}=1.
$$
If $s\in [a,1)$, then $f(s)\in ({\tau_2},b]$ and
$$
\varphi_1(s)+\varphi_2(f(s))=\frac{s+1-a-b}{2-a-b}+\frac{1-f^{-1}(f(s))}{2-a-b}
=\frac{s+1-a-b+1-s}{2-a-b}=1,
$$
since $f^{-1}(f(s))=s$ for $s\in [a,1)$. Finally, if $s=1$, then $f(1)={\tau_2}$ and $\varphi_2({\tau_2})=0$, and thus again $\varphi_1(1)+\varphi_2(f(1))=1$.
This concludes the proof in the general situation.

If $K$ is symmetric with respect to $x_1=x_2$, then $f$ is its own inverse, $a=b$, and ${\tau_1}={\tau_2}$. Moreover, only the third case in the above argument has to be considered, since $f_l'(a)\ge -1$ and $f_r'(a)\le -1$ at the unique $a\in (0,1)$ such that $f(a)=a$. Thus we get $\varphi_1=\varphi_2$.
\end{proof}

Note that with $f$ as in the previous proof, we have $x_2=f(x_1)$, $0\le x_1\le 1$, and can therefore represent $\partial K\cap [0,1]^2$ by the equation $g_1(x_1)+g_2(x_2)=1$, where $g_1(x_1)=1-f(x_1)$, $0\le x_1\le 1$, and $g_2(x_2)=x_2$ are convex with $g_1(0)=g_2(0)=0$.  However, in general $g_1$ cannot be extended to a function in $\Phi$ since its derivative at $x_1=1$ may be infinite.

If $K=[-1,1]^2$, the representation (\ref{Jphi2}) fails when $(x_1,x_2)=(1,1)$.

The functions $\varphi_j$, $j=1,2$, in Theorem~\ref{inverseTh4.4} are not unique, in general. For example, if $K$ is the unit disk, then the proof of Theorem~\ref{inverseTh4.4} with $f(t)=\sqrt{1-t^2}$, $0\le t\le 1$, provides the representation \eqref{Jphi2} with $\varphi_1=\varphi_2=\varphi_0$ given by
$$
\varphi_0(s)=\begin{cases}
\left(1 - \sqrt{1 -s^2}\right)/(2 - \sqrt{2}),& {\text{ if $0\le s\le 1/\sqrt{2}$,}}\\
(s + 1 - \sqrt{2})/(2 - \sqrt{2}),& {\text{ if $s>1/\sqrt{2}$}},
\end{cases}
$$
but another is given by $\varphi_0(s) = s^2$, $0\le s\le 1$.

The following example shows that Theorem~\ref{inverseTh4.4} does not generally hold for $n>2$.

\begin{ex}\label{ex4.4W}
{\em Let $D$ be the $(n-1)$-dimensional unit ball in the coordinate plane $\{x_n=0\}$ in $\R^n$ and let $K=\conv\{D,\pm e_n\}$.  Then $K$ is a 1-unconditional double cone containing $e_1,\dots,e_n$ in its boundary.  Suppose that there are $\varphi_j\in \Phi$, $j=1,\dots,n$, such that $\partial K\cap [0,1]^n$ is given by the equation
\begin{equation}\label{Jphi2ex}
\varphi(x_1,\dots,x_n)=\varphi_1(x_1)+\cdots+\varphi_n(x_n)=1,
\end{equation}
for $x_1,\dots,x_n\ge 0$.  Let $i\in \{1,\dots,n-1\}$ and let $S$ be the subspace spanned by $e_i$ and $e_n$.  Since $\varphi_j(0)=0$, $j=1,\dots,n$, it follows from (\ref{Jphi2ex}) that the set $\partial K\cap [0,1]^n\cap S$ is given by the equation
\begin{equation}\label{Jphi2ex2}
\varphi_i(x_i)+\varphi_n(x_n)=1,
\end{equation}
for $x_i,x_n\ge 0$.  Now
$$\partial K\cap [0,1]^n\cap S=[e_i,e_n]=\{(1-t)e_i+te_n: 0\le t\le 1\}.$$
We have $\varphi_i(1)=\varphi_n(1)=1$, so if $\varphi_i(1/2)<1/2$, then (\ref{Jphi2ex2}) implies $\varphi_n(1/2)>1/2=\varphi_n(1)/2$, contradicting the convexity of $\varphi_n$.  Since $\varphi_i$ is convex, this yields $\varphi_i(1/2)=1/2$ and hence $\varphi_i(t)=t$ for $0\le t\le 1$.  Similarly, we obtain $\varphi_n(t)=t$ for $0\le t\le 1$.  Now (\ref{Jphi2ex}) becomes
$x_1+\cdots+x_n=1$, but when $x_n=0$ this contradicts the fact that $\partial K\cap [0,1]^n\cap e_n^{\perp}=S^{n-1}\cap [0,1]^n\cap e_n^{\perp}$.}
\end{ex}

In the previous example, the convexity of the functions $\varphi_j$, $j=1,\dots,n$, is used in an essential way.  Indeed, when $n=3$, the choice $\varphi_1(x_1)=x_1^2$, $\varphi_2(x_2)=x_2^2$, and $\varphi_3(x_3)=2x_3-x_3^2$, $x_1, x_2, x_3\in [0,1]$, in (\ref{Jphi2ex}), for example, describes the boundary of $K\cap [0,1]^3$, but then $\varphi_3$ is not convex. Example~\ref{ex4.4W} may also be viewed in the context of Kolmogorov's superposition theorem (see, for example, \cite[Chapter~11]{Lor66}), which arose from Hilbert's thirteenth problem.

\begin{cor}\label{corinverseTh4.4}
Let $M$ be a $1$-unconditional convex body in $\R^m$ that contains $e_1,\ldots,e_m$ in its boundary. Then there is a
$\varphi\in\Phi_m$ such that $\oplus_M=+_{\varphi}$ as operations $\left({\mathcal{K}}^n_s\right)^2\to {\mathcal{K}}^n_s$ and $\oplus_{M\cap [0,\infty)^2}=+_{\varphi}$ as operations $\left({\mathcal{K}}^n_o\right)^2\rightarrow {\mathcal{K}}^n_o$. Moreover,

\smallskip

\noindent {\rm{(i)}} if $m\ge 2$, the function $\varphi$ can be chosen to be homogeneous of degree 1;

\noindent {\rm{(ii)}}  if $m=2$, $\varphi$ can be defined by \eqref{phimax} if $M=\conv\{\pm e_1,\pm e_2\}$ and by \eqref{Jphi2}, for some $\varphi_1,\varphi_2\in \Phi$, otherwise.  If $M$ is also symmetric with respect to $x_1=x_2$, then the latter holds with $\varphi_1=\varphi_2=\varphi_0$.
\end{cor}

\begin{proof}
The assumptions on $M$ imply that $K=M^\circ$ satisfies the hypotheses of Theorem \ref{inverseTh4.4general} and Theorem~\ref{inverseTh4.4}.  Then the functions $\varphi\in \Phi_m$ provided by these theorems have the property that $\partial M^{\circ}\cap [0,1]^m$ is given by \eqref{Th4.4neweq}, (\ref{phimax}), or (\ref{Jphi2}), respectively.  Therefore $M^\circ =J_{\varphi}$ and hence $M=J_{\varphi}^\circ$, where $\varphi$ in (\ref{Jphi}) takes the special forms in the statement of the corollary.  The result follows from Theorem~\ref{Orthm2}.
\end{proof}

\begin{rem}\label{remOMequiv}
{\em The assumption in Corollary~\ref{corinverseTh4.4} that $M$ contains $e_1,\ldots,e_m$ in its boundary can be removed by working with  $\overline{\Phi}_m$ as defined in Section~\ref{prelim} instead of $\Phi_m$.  In this case, in Corollary~\ref{corinverseTh4.4}(ii), we have that
$\varphi$ can be defined by $\varphi(x_1,x_2)=\max\{ax_1, bx_2\}$ if
$M=\conv\{\pm ae_1,\pm be_2\}$ for some $a,b> 0$, and by \eqref{Jphi2}, for some $\varphi_1,\varphi_2\in \overline{\Phi}_1$, otherwise.
}
\end{rem}

The previous theorems allow us to find necessary and sufficient conditions for the commutativity or associativity of $+_{\varphi}$.

\begin{thm}\label{thmcomm}
The operation $+_{\varphi}:\left({\mathcal{K}}^n_s\right)^2\rightarrow {\mathcal{K}}^n_s$ (or $+_{\varphi}:
\left({\mathcal{K}}^n_o\right)^2\rightarrow {\mathcal{K}}^n_o$) is commutative if and only if it can be defined by \eqref{Orldef}, where $\varphi$ is given by \eqref{phimax} or by \eqref{Jphi2}, where $\varphi_1=\varphi_2$.
\end{thm}

\begin{proof}
Suppose that  $+_{\varphi}:\left({\mathcal{K}}^n_s\right)^2\rightarrow {\mathcal{K}}^n_s$ (or $+_{\varphi}:
\left({\mathcal{K}}^n_o\right)^2\rightarrow {\mathcal{K}}^n_o$) is commutative.
By Theorem~\ref{Orthm2}, there is a 1-unconditional convex body $M\subset[-1,1]^2$ (namely, $J_{\varphi}^\circ$), containing $e_1$ and $e_2$ in its boundary, such that $+_{\varphi}={\oplus}_M$ (or $+_{\varphi}={\oplus}_{M\cap [0,\infty)^2}$, respectively).  Let $(a,b)\in M$.  If $K=[-e_1, e_1]$ and $L=[-e_2,e_2]$, then since $\oplus_M$ is commutative, we have
$$
ae_1+be_2\in K\oplus_M L= L\oplus_M K.
$$
Therefore there are $s,t\in [-1,1]$ and $(c,d)\in M$ such that
$$ae_1+be_2=cse_2+dte_1$$
and hence $cs=b$ and $dt=a$. If $s,t\neq 0$, it follows that $(b/s,a/t)\in M$ and since $M$ is 1-unconditional, $(b/|s|,a/|t|)\in M$.  Again using the fact that $M$ is 1-unconditional, we conclude that $(b,a)\in M$ because $|s|,|t|\le 1$.  If $s=0$, we must have $b=0$ and the same argument shows that $(0,a)\in M$. The case when $t=0$ is dealt with similarly.

This proves that $M$ is symmetric with respect to $x_1=x_2$.  The desired conclusion now follows from Corollary~\ref{corinverseTh4.4}.

The converse is an immediate consequence of (\ref{Orldef}).
\end{proof}

\begin{thm}\label{thmassoc}
The operation $+_{\varphi}:\left({\mathcal{K}}^n_s\right)^2\rightarrow {\mathcal{K}}^n_s$ (or $+_{\varphi}:
\left({\mathcal{K}}^n_o\right)^2\rightarrow {\mathcal{K}}^n_o$) is associative if and only if $+_{\varphi}=+_p$, that is, $+_{\varphi}$ is $L_p$ addition, for some $1\le p\le\infty$.
\end{thm}

\begin{proof}
Suppose that $+_{\varphi}:\left({\mathcal{K}}^n_s\right)^2\rightarrow {\mathcal{K}}^n_s$ (or $+_{\varphi}:
\left({\mathcal{K}}^n_o\right)^2\rightarrow {\mathcal{K}}^n_o$). Then (\ref{nnnn}) holds.  The associativity of $+_{\varphi}$ allows the argument in the proof of \cite[Theorem~7.9]{GHW} to be applied, with $M$ there replaced by $J_{\varphi}^{\circ}$, and this shows that the convex body $J_{\varphi}^{\circ}$ must be the unit ball in $l_p^2$ for some $1\le p\le \infty$.  Hence  $+_{\varphi}=+_p$, for some $1\le p\le \infty$.
\end{proof}

\section{Extensions to arbitrary sets}\label{Extension}
In Section \ref{Orlicz}, an Orlicz addition $+_{\varphi}$, for $\varphi\in\Phi_m$ , is defined between convex sets containing the origin.
Now we extend $+_{\varphi}$ to an operation between general compact sets.

Recall the definition \eqref{Jphi} of the 1-unconditional convex body $J_{\varphi}\subset[-1,1]^m$ for a given function $\varphi\in\Phi_m$.
By Theorem \ref{Orthm2}, we know that $+_\varphi(K_1,\ldots,K_m)=\oplus_{J_\varphi^\circ \cap[0,\infty)^m}(K_1,\ldots,K_m)$,
for all $K_1,\ldots,K_m\in\K^n_o$. Therefore we can define
\begin{equation}\label{ExtOrldeff}
+_\varphi(K_1,\ldots,K_m)=\oplus_{J_\varphi^\circ \cap[0,\infty)^m}(K_1,\ldots,K_m),
\end{equation}
for arbitrary sets $K_1,\dots,K_m$ in $\R^n$, where the right-hand side is given by (\ref{Mndef}) with $M=J_\varphi^\circ \cap[0,\infty)^m$.

Orlicz addition can be extended to an operation $+_{\varphi}:\left({{\K}}^n\right)^m\to {{\K}}^n_o$, equivalently, by setting
\begin{equation}\label{ExtOrldef}
+_{\varphi}(K_1,\dots,K_m)=+_{\varphi}\left(\conv\{K_1,o\},\dots,
\conv\{K_m,o\}\right),
\end{equation}
for $K_j\in {{\K}}^n$, $j=1,\dots,m$, where the right-hand side is defined by (\ref{Orldef}).  Indeed, in view of the fact that $J_{\varphi}^{\circ}$ is 1-unconditional, we have
\begin{eqnarray*}
h_{J_\varphi^\circ \cap[0,\infty)^m}\left(h_{K_1}(x),\dots,h_{K_m}(x)\right)
&=&
h_{J_\varphi^\circ \cap[0,\infty)^m}\left(
\max\{h_{K_1}(x),0\},\dots,\max\{h_{K_m}(x),0\}\right)\\
&=&
h_{J_\varphi^\circ \cap[0,\infty)^m}\left(
h_{\conv\{K_1,o\}}(x),\dots,h_{\conv\{K_m,o\}}(x)\right),
\end{eqnarray*}
for all $x\in \R^n$, so (\ref{ExtOrldef}) agrees with (\ref{ExtOrldeff}).

To illustrate, consider the $L_p$ sum, $1\le p\le \infty$, of two sets.  In this case we have
$$J_\varphi^\circ \cap[0,\infty)^2=\{(x_1,x_2)\in [0,1]^2: x_1^{p'}+x_2^{p'}\le 1\},$$
where $1/p+1/p'=1$.  When $K,L\in {\K}^n$, this leads to the extension of $L_p$ addition given in \cite[Example~6.7]{GHW}. Denoting this extension also by $+_p$, we then have $K+_pL=\conv\{K,o\}+_p\conv\{L,o\}$, and, in particular, $K+_{\infty}L=\conv\{K\cup L,o\}$, which agrees with (\ref{ExtOrldef}).

The operation $+_{\varphi}:\left({{\K}}^n\right)^m\to {{\K}}^n_o$ defined by (\ref{ExtOrldeff}) or (\ref{ExtOrldef}) is monotonic, continuous,  $GL(n)$ covariant, and projection covariant.  This follows from a straightforward modification of Theorem~\ref{Orthm1}, using the fact that the map taking $K\in {\K}^n$ to $\conv\{K,o\}$ is monotonic, continuous in the Hausdorff metric, and satisfies $A\left(\conv\{K,o\}\right)=\conv\{AK,o\}$ for each linear map $A:\R^n\to \R^n$.  However, the extended operation no longer has the identity property, in general.

A different extension is possible when $m=2$. If $J_{\varphi}^{\circ}=[-1,1]^2$, then by Theorem~\ref{Orthm2}, $+_{\varphi}:\left({\mathcal{K}}^n_o\right)^2\rightarrow {\mathcal{K}}^n_o$ is Minkowski addition and already makes sense for arbitrary sets $K$ and $L$ via the formula
\begin{equation}\label{Mextens}
K+_{\varphi}L=K+L=\{x+y:x\in K, y\in L\}.
\end{equation}

Otherwise, we have $J_{\varphi}^{\circ}\neq [-1,1]^2$
and can apply Theorem~\ref{inverseTh4.4} with $K=J_{\varphi}^{\circ}$ to conclude that
\begin{equation}\label{psipn}
J_{\varphi}^{\circ}\cap [0,\infty)^2=\{(x_1,x_2)\in [0,1]^2:\psi_1(x_1)+\psi_2(x_2)
\le 1\},
\end{equation}
for some $\psi_1, \psi_2\in\Phi$.  If $\tau_j=\tau(\psi_j)<1$ is defined by (\ref{tau}) with $\varphi$ replaced by $\psi_j$, we denote by $\widehat{\psi_j}$ the restriction of $\psi_j$ to $[\tau_j,1]$, for $j=1,2$.  Then $\widehat{\psi_j}:[\tau_j,1]\to [0,1]$, $j=1,2$, is a bijection and for arbitrary subsets $K$ and $L$ of $\R^n$, we define
\begin{equation}\label{defarb}
K+_{\varphi}L=\{\widehat{\psi_1}^{-1}(1-t)x+
\widehat{\psi_2}^{-1}(t)y:x\in K, y\in L, 0\le t\le 1\}.
\end{equation}
This is the Orlicz analog of the $L_p$ case (\ref{LYZLp}).  We remarked after (\ref{LYZLp}) that the right-hand side is not generally convex when $K, L\in {\K}^n$ and hence this extension of Orlicz addition is different from the one above.  The following result generalizes that in \cite[Lemma~1.1]{LYZ}, which is the corresponding one for $L_p$ addition.

\begin{thm}\label{thmarb}
If $K, L\in{\mathcal K}^n_{o}$, the definition of $+_{\varphi}$ via \eqref{defarb} agrees with the one via \eqref{Orldef}.
\end{thm}

\begin{proof}
Let $K, L\in{\mathcal K}^n_{o}$. We claim that the set $K+_{\varphi}L$ defined by \eqref{defarb} equals $K\oplus_M L$, where $M=J_{\varphi}^{\circ}\cap [0,\infty)^2$ is given by (\ref{psipn}).  This will suffice to prove the result, since we know from Theorem~\ref{Orthm2} that $K\oplus_M L$ equals the set $K+_{\varphi}L$ defined by \eqref{Orldef}.

To prove the claim, let $E$ be $K+_{\varphi}L$ as defined by \eqref{defarb}.  We have $K\oplus_M L\in{\mathcal K}^n_{o}$ and
$$E=\{\widehat{\psi_1}^{-1}(1-t)x+\widehat{\psi_2}^{-1}(t)y:x\in K, y\in L, 0\le t\le 1\}\subset K\oplus_M L,$$
since it follows from (\ref{psipn}) that $\left(\widehat{\psi_1}^{-1}(1-t),\widehat{\psi_2}^{-1}(t)\right)\in M$.  Suppose that $z=ax+by\in K\oplus_ML$, where $x\in K$, $y\in L$, and $(a,b)\in M$, so that $\psi_1(a)+\psi_2(b)\le 1$. If $a=b=0$, then $z=o\in E$. If $a+b>0$, then there is a unique $\alpha>0$ such that
\begin{equation}\label{alphi}
\psi_1\left(\frac{a}{\alpha}\right)+
\psi_2\left(\frac{b}{\alpha}\right)=1.
\end{equation}
(The uniqueness is a consequence of $\psi_1,\psi_2\in \Phi$; it can easily be proved directly but was already used more generally in the definition (\ref{Orldef}).) Let
\begin{equation}\label{tt}
t=\psi_2\left(\frac{b}{\alpha}\right).
\end{equation}
From (\ref{alphi}) it follows that $0\le t\le 1$.  By (\ref{alphi}) and the fact that $\psi_1(a)+\psi_2(b)\le 1$, we must also have $0<\alpha\le 1$.  Then $\alpha x\in K$ and $\alpha y\in L$.

If $a/\alpha\ge \tau_1$ and $b/\alpha\ge\tau_2$, then from (\ref{alphi}) and (\ref{tt}) we get
$b/\alpha=\widehat{\psi_2}^{-1}(t)$ and $a/\alpha=\widehat{\psi_1}^{-1}(1-t)$, and hence
$$z=ax+by=\widehat{\psi}^{-1}(1-t)\alpha x+\widehat{\psi}^{-1}(t)\alpha y\in E.$$
If $b/\alpha< \tau_2$, then $t=\psi_2(b/\alpha)=0$. Then $\psi_1(a/\alpha)=1$ and thus $a=\alpha$.  Using $\widehat{\psi_1}^{-1}(1-0)=1$, $\widehat{\psi_2}^{-1}(0)=\tau_2$, and $b/\tau_2<\alpha\le 1$, we obtain
$$z=ax+by=\widehat{\psi_1}^{-1}(1-0)\alpha x+\widehat{\psi_2}^{-1}(0)\frac{b}{\tau_2} y\in E.$$
A similar argument applies if $a/\alpha< \tau_1$. This proves the claim.
\end{proof}

\section{Brunn-Minkowski-type inequalities}\label{Minequality}

In this section we establish Brunn-Minkowski-type inequalities for $M$-addition and Orlicz addition of compact sets.

\begin{lem}\label{BrunnMinkowski} If $M\in {\mathcal{C}}^m$ and $K_1,\ldots,K_m\in {\mathcal{C}}^n$, then
\begin{equation}\label{BMI}
{V}(\oplus_M (K_1,\ldots,K_m))^{1/n} \ge \sum_{j=1}^m|a_j| {V}(K_j)^{1/n} ,
\end{equation}
for all $(a_1,\ldots,a_m)\in M$.  If $V(K_j)>0$ for $j=1,\ldots,m$ and equality holds in \eqref{BMI} for some $(a_1,\ldots,a_m)\in M$ with $a_j\neq 0$, $j=1,\ldots,m$, then $K_1,\ldots,K_m$ are homothetic convex bodies.
\end{lem}

\begin{proof}
Let $(a_1,\ldots,a_m)\in M$.  From \eqref{Mnaltdef}, we obtain
$$
{V}(\oplus_M (K_1,\ldots,K_m)) \ge {V}\left(\sum_{j=1}^m a_j K_j\right).
$$
The Brunn-Minkowski inequality for compact sets (see, for example, \cite{Gar02}) yields
\begin{equation}\label{BM}
{V}\left(\sum_{j=1}^m a_j K_j\right)^{1/n}\ge \sum_{j=1}^m|a_j| {V}(K_j)^{1/n},
\end{equation}
and (\ref{BMI}) follows.

Suppose that ${V}(K_j)>0$ for $j=1,\ldots,m$ and equality holds in (\ref{BMI}) for some $(a_1,\ldots,a_m)\in M$ with $a_j\neq 0$, $j=1,\ldots,m$.  Then equality holds in (\ref{BM}).  When $m=2$, the equality condition for the Brunn-Minkowski inequality for compact sets (see \cite[p.~363]{Gar02}) implies that $K_1$ and $K_2$ are homothetic convex bodies.  Suppose that when $m=p\ge 2$, equality in (\ref{BM}) implies that $K_1,\dots,K_p$ are homothetic convex bodies.  If $m=p+1$, let $L_j=a_jK_j$, $j=1,\dots,p+1$.  Then
$${V}\left(\sum_{j=1}^{p+1} L_j\right)^{1/n}\ge
{V}\left(\sum_{j=1}^{p} L_j\right)^{1/n}+V(L_{p+1})^{1/n}\ge \sum_{j=1}^{p+1}{V}(L_j)^{1/n},$$
so if equality holds in (\ref{BM}) when $m=p+1$, we conclude that $L_1,\dots,L_p$ are homothetic and that $L_{p+1}$ is homothetic to $L_1+\cdots+L_p$.  This means that $L_1,\dots,L_{p+1}$ and hence $K_1,\dots,K_{p+1}$ are homothetic convex bodies.  The equality condition then follows by induction on $m$.
\end{proof}

The assumption that $a_j\neq 0$, $j=1,\ldots,m$, for some $(a_1,\ldots,a_m)\in M$ cannot be omitted in the equality condition for Lemma~\ref{BrunnMinkowski}, even when $m=2$.  For example, let $M=[e_1,e_2]$ and let $K,L\in {\mathcal{K}}^n_o$.  Then, as was mentioned in Section~\ref{Maddition}, $K\oplus_M L=K+_{\infty} L=\conv(K\cup L)$ and equality holds in (\ref{BMI}) with $m=2$, $K_1=K$, and $K_2=L$, when $(a_1,a_2)=e_2$ (or $(a_1,a_2)=e_1$) if and only if $K\subset L$ (or $L\subset K$, respectively), so $K$ and $L$ need not be homothetic.

Recall that the support set $F(K,x)$ of a compact convex set $K$ with outer normal vector $x\neq o$ is defined by (\ref{suppset}).

\begin{cor}\label{corBrunnMinkowski}
Let $M\in {{\mathcal{C}}^m}$ and let $K_1,\ldots,K_m\in {{\mathcal{C}}^n}$. Then
\begin{equation}\label{BMI2}
{V}(\oplus_M (K_1,\ldots,K_m))^{1/n} \ge h_{\conv M}\left({V}(K_1)^{1/n},\ldots,{V}(K_m)^{1/n}\right).
\end{equation}
If ${V}(K_j)>0$ for $j=1,\ldots,m$, $M\cap F(\conv M,x)\not\subset \cup_{j=1}^me_j^{\perp}$ for all $x\in (0,\infty)^m$, and equality holds
in \eqref{BMI2}, then $K_1,\ldots,K_m$ are homothetic convex bodies.
\end{cor}

\begin{proof}
Clearly,
$$h_{\conv M}(x)=\max\{x\cdot y: y\in M\},$$
for $x\in\R^m$. The inequality (\ref{BMI}) is optimal when we choose $a_1,\ldots,a_m$ so that the right-hand side is as large as possible, i.e., when the right-hand side equals
$$\max\left\{(a_1,\ldots,a_m)\cdot \left( {V}(K_1)^{1/n},\dots,{V}(K_m)^{1/n}
\right):(a_1,\ldots,a_m)\in M\right\}$$
$$=h_{\conv M}\left({V}(K_1)
^{1/n},\dots,{V}(K_m)^{1/n}\right).$$
This proves (\ref{BMI2}).

Suppose that  ${V}(K_j)>0$ for $j=1,\ldots,m$, $M\cap F(\conv M,x)$ is not contained in the union $\cup_{j=1}^me_j^{\perp}$ of the coordinate hyperplanes  for all $x\in (0,\infty)^2$, and equality holds in (\ref{BMI2}).  Then $x=\left({V}(K_1)
^{1/n},\ldots,{V}(K_m)^{1/n}\right)\in (0,\infty)^m$ and the condition $M\cap F(\conv M,x)\not\subset \cup_{j=1}^me_j^{\perp}$ means that there is an $(a_1,\ldots,a_m)\in M$ with $a_j\neq 0$, $j=1,\ldots,m$, for which
\begin{equation}\label{abh}
(a_1,\ldots,a_m)\cdot \left( {V}(K_1)^{1/n},\ldots,{V}(K_m)^{1/n}
\right)=h_{\conv M}\left({V}(K_1)
^{1/n},\ldots,{V}(K_m)^{1/n}\right).
\end{equation}
Now by (\ref{BMI2}) and (\ref{abh}), we must have $a_j\ge 0$, $j=1,\ldots,m$, and equality in (\ref{BMI}) for this $(a_1,\ldots,a_m)$, so the equality condition follows from that in Lemma~\ref{BrunnMinkowski}.
\end{proof}

Consider the case $m=2$ of the previous corollary, with $K_1=K$ and $K_2=L$.  Suppose that ${V}(K){V}(L)>0$.  If $M=\{(1,1)\}$, then $K\oplus_M L=K+L$ and equality holds in Lemma~\ref{BrunnMinkowski}, under the conditions stated there, or in (\ref{BMI2}), precisely when $K$ and $L$ are homothetic convex bodies.  On the other hand, if $M$ is such that $K\oplus_M L=K+_pL$ for some $1<p<\infty$, then equality holds in Lemma~\ref{BrunnMinkowski}, under the conditions stated there, or in (\ref{BMI2}), if and only if $K,L\in {\mathcal{K}}^n_o$ and are dilatates of each other (see Corollary~\ref{LpBrunnMinkowski} below).  Moreover, if $M=[e_1,e_2]$, then there is equality in (\ref{BMI2}) if and only if $K\subset L$ or $L\subset K$.  (This also shows that the condition that $M\cap F(\conv M,x)\not\subset l_{e_1}\cup l_{e_2}$ for all $x\in (0,\infty)^2$ cannot be removed in the statement of Corollary~\ref{corBrunnMinkowski}.)
A general discussion of equality in (\ref{BMI}) and (\ref{BMI2}) appears to be complicated and we focus on a special case of particular interest.

\begin{lem}\label{eqcase2}
Let $M\in {{\mathcal{K}}^2}$ be contained in $[0,1]^2$ and contain $e_1$ and $e_2$.  Suppose further that each point in $\cl(\partial M\cap (0,1)^2)$ is contained in a unique supporting line to $M$ with outer normal vector in $[0,\infty)^2\setminus \{o\}$. Let $K,L\in {\mathcal{C}}^n$ be such that ${V}(K){V}(L)>0$ and either equality holds in \eqref{BMI} for some $(a,b)\in M$  or equality holds in \eqref{BMI2}. Then $K,L\in {\mathcal{K}}^n_o$ and are dilatates of each other.
\end{lem}

\begin{proof}
Suppose that equality holds in (\ref{BMI}) for some $(a,b)\in M$.
Let $x_0=\left({V}(K)^{1/n},{V}(L)^{1/n}\right)\in (0,\infty)^2$.  Then
\begin{equation}\label{very1}
V(K\oplus_M L)^{1/n}\ge h_M(x_0)\ge (a,b)\cdot x_0= a V(K)^{1/n}+b V(L)^{1/n},
\end{equation}
where the left-hand inequality comes from Corollary~\ref{corBrunnMinkowski}.  Our assumption implies that equality holds throughout (\ref{very1}) and hence $h_M(x_0)= (a,b)\cdot x_0$. This shows that $x_0\in (0,\infty)^2$ is an outer normal vector to $M$ at $(a,b)\in\partial M$.  Then the properties of $M$ we assume imply that
$$
F(M,x_0)\cap ([e_1,e_1+e_2]\cup[e_2,e_1+e_2])=\emptyset,
$$
since any point in the left-hand side is contained in $\cl(\partial M\cap (0,1)^2)$ and also in more than one supporting line to $M$ with outer normal vector in $[0,\infty)^2\setminus\{o\}$. We conclude that $a,b\in (0,1)$ and that the assumptions in the equality condition in Corollary~\ref{corBrunnMinkowski} are satisfied, so $K$ and $L$ are homothetic convex bodies.

Since $K\oplus_M L\supset aK+bL$ and
$$
V(K\oplus_M L)^{1/n}\ge V(aK+bL)^{1/n}\ge aV(K)^{1/n}+bV(L)^{1/n},
$$
we must again have equality throughout, and therefore $K\oplus_M L=aK+bL$. In particular, we have $rK+sL\subset aK+bL$ for all $(r,s)\in M$.

We claim that $o\in K\cap L$. To see this, note that since $e_1,e_2\in M$, we have $K,L\subset aK+bL$. Therefore $(1-a)h_K\le bh_L$ and $(1-b)h_L\le ah_K$, where $a,b\in (0,1)$. It follows that $h_K\le c_0 h_K$ and $h_L\le c_0 h_L$ with $c_0=ab(1-a)^{-1}(1-b)^{-1}$, and hence $(c_0-1)h_K\ge 0$ and $(c_0-1)h_L\ge 0$.
If $c_0>1$, then $h_K\ge 0$ and $h_L\ge 0$, and therefore $o\in K\cap L$.
If $c_0=1$, then $a+b=1$; since $(a,b)\in \partial M$ and $e_1,e_2\in M$, this yields $[e_1,e_2]\subset \partial M$, contradicting our assumptions on $M$.
If $c_0<1$, then $h_K\le 0$ and $h_L\le 0$, which implies that $K=L=\{o\}$, again contradicting our assumptions. This proves the claim.

We now know that $K,L$ are homothetic convex bodies containing the origin.  By (\ref{nn}) and the fact that $K\oplus_M L=aK+bL$, we have
$$h_M\left(h_K(x),h_L(x)\right)=(a,b)\cdot \left(h_K(x),h_L(x)\right),$$
for all $x\in \R^n$, where $(a,b)\in \partial M\cap (0,1)^2$ and $(h_K(x),h_L(x))\in [0,\infty)^2$.  Therefore $(a,b)$ lies in the unique supporting hyperplane to $M$ with outer normal vector $(h_K(x),h_L(x))$, for every $x\in \R^n\setminus\{o\}$ for which $(h_K(x),h_L(x))\neq o$.  By our hypothesis on $\partial M\cap (0,1]^2$, such outer normal vectors are unique, up to multiplication by a constant factor. It follows that there is a $u\in S^{n-1}$ such that for each $x\in \R^n\setminus\{o\}$, there is a $c(x)\ge 0$ such that
$$(h_K(x),h_L(x)) = c(x)(h_K(u),h_L(u)).$$
Moreover, $h_K(u)\neq 0$ and $h_L(u)\neq 0$, since $e_1,e_2\in M$ and $a,b\in (0,1)$. But then
$$
h_K(x)=\frac{h_K(u)}{h_L(u)}c(x)h_L(u)=\frac{h_K(u)}{h_L(u)}h_L(x),
$$
i.e., $h_K$ and $h_L$ are the same up to a nonzero constant multiple, so $K$ and $L$ are dilatates of each other. This establishes the first part of the lemma.

Now suppose that equality holds in (\ref{BMI2}), i.e., we have
$$
V(K\oplus_M L)^{1/n}=h_M(V(K)^{1/n},V(L)^{1/n}).
$$
Then there exists a point $(a,b)\in M$ such that
$$
V(K\oplus_M L)^{1/n}=(a,b)\cdot \left(V(K)^{1/n},V(L)^{1/n}\right)=aV(K)^{1/n}+bV(L)^{1/n},
$$
and thus equality holds in (\ref{BMI}). The assertion now follows from the first part of the lemma.
\end{proof}

\begin{cor}\label{corBrunnMinkowski2new}
If $\varphi \in\Phi_2$ and $K,L\in {{\mathcal{C}}^n}$, then
\begin{equation}\label{Orlnew}
{V}(K+_{\varphi}L)^{1/n} \ge
h_{J_{\varphi}^{\circ}}\left({V}(K)^{1/n},
{V}(L)^{1/n}\right).
\end{equation}
When $\varphi$ is strictly convex and ${V}(K){V}(L)>0$, equality holds if and only if $K,L\in {{\mathcal{K}}^n_o}$ and are dilatates of each other.
\end{cor}

\begin{proof}
If $J_{\varphi}^{\circ}=[-1,1]^2$, then $+_{\varphi}:\left({\mathcal{K}}^n_o\right)^2\rightarrow {\mathcal{K}}^n_o$ is Minkowski addition and $K+_{\varphi}L=K+L$ is given by (\ref{Mextens}) when $K,L\in {{\mathcal{C}}^n}$.  Therefore  (\ref{Orlnew}) is just the usual Brunn-Minkowski inequality for compact sets.

Otherwise, we have defined $K+_{\varphi}L$ for arbitrary sets $K$ and $L$ by (\ref{defarb}) and shown that it agrees with our previous definition when $K,L\in {{\mathcal{K}}^n_o}$. Let
$$M=\left\{\left(\widehat{\psi_1}^{-1}(1-t),\widehat{\psi_2}^{-1}(t)\right): 0\le t\le 1\right\},$$
where $\widehat{\psi_1}$ and $\widehat{\psi_2}$ are as in (\ref{defarb}). Then $M\in {{\mathcal{C}}^2}$ and
$$h_{\conv M}(s,t)=h_{J_{\varphi}^{\circ}}(s,t),$$
for $s,t\ge 0$.  By Corollary~\ref{corBrunnMinkowski}, we have
$${V}(K\oplus_M L)^{1/n} \ge h_{J_{\varphi}^{\circ}}\left({V}(K)^{1/n},
{V}(L)^{1/n}\right),$$
for all $K,L\in {{\mathcal{C}}^n}$.  The result follows, since (\ref{defarb}) implies that $K\oplus_M L=K+_{\varphi}L$ for arbitrary $K$ and $L$.

Suppose that $\varphi$ is strictly convex.  Then  $J_{\varphi}$ is strictly convex.  By \cite[p.~107]{Sch93}, $h_{J_{\varphi}}$ is smooth (i.e., of class $C^1$).  Therefore $\rho_{J_{\varphi}}^{\circ}=1/h_{J_{\varphi}}$ and hence $J_{\varphi}^{\circ}$ itself is also smooth.  From this we conclude that
$J_{\varphi}^{\circ}\cap [0,\infty)^2$ satisfies the hypotheses on $M$ in Lemma~\ref{eqcase2} and the stated equality condition follows from that lemma.
\end{proof}

The Orlicz sum in \eqref{Orlnew} was considered to be defined via (\ref{defarb}). The result remains true (by a similar argument), if we
take \eqref{ExtOrldeff} to define the Orlicz sum, and then the inequality holds for $m\ge 2$. A similar remark applies to the
next result.

\begin{cor}\label{OBMnewer}
Let $\varphi\in \Phi_2$.  If $K, L\in {\mathcal{C}}^n$ and ${V}(K){V}(L)>0$, then
\begin{equation}\label{OBMnn}
1\ge \varphi\left(\left(\frac{{V}(K)}
{{V}(K+_{\varphi}L)}\right)^{1/n},\left(\frac{{V}(L)}
{{V}(K+_{\varphi}L)}\right)^{1/n}\right).
\end{equation}
When $\varphi$ is strictly convex, equality holds if and only if $K,L\in {{\mathcal{K}}^n_o}$ and are dilatates of each other.
\end{cor}

\begin{proof}
Using (\ref{eq2}), we obtain
\begin{equation}\label{assist2}
\varphi\left(\frac{{V}(K)^{1/n}}{h_{J_{\varphi}^{\circ}}\left({V}(K)^{1/n},
{V}(L)^{1/n}\right)},\frac{{V}(L)^{1/n}}
{h_{J_{\varphi}^{\circ}}\left({V}(K)^{1/n},{V}(L)^{1/n}\right)}\right)=1.
\end{equation}
(This is because we can take $s=V(K)^{1/n}$ and $t=V(L)^{1/n}$, and it was proved after (\ref{eq2}) that we have $M=J_{\varphi}^{\circ}$.)  Since $\varphi$ is increasing in each variable, (\ref{OBMnn}) follows from (\ref{Orlnew}) and (\ref{assist2}). If $\varphi$  is strictly convex, then it is strictly increasing in each variable and the equality condition in (\ref{OBMnn}) follows directly from (\ref{assist2}) together with (\ref{Orlnew}) and its equality condition.
\end{proof}

It is also true that (\ref{Orlnew}) and its equality condition follows directly from (\ref{assist2}) and (\ref{OBMnn}) and its equality condition, so the two inequalities are equivalent.

Recall that the $L_p$ sum of arbitrary sets $K$ and $L$ in $\R^n$ can be defined by (\ref{LYZLp}).  The following inequality was first proved by Firey \cite{Fir62} for convex bodies containing the origin in their interiors.

\begin{cor}\label{LpBrunnMinkowski}  {\rm(Lutwak, Yang, and Zhang \cite{LYZ}.)}
Let $p>1$. If $K,L\in {\mathcal{C}}^n$, then
$$
{V}(K+_pL)^{p/n} \ge {V}(K)^{p/n}+
{V}(L)^{p/n}.
$$
If ${V}(K){V}(L)>0$, then equality holds if and only if $K,L\in {\mathcal{K}}^n_o$ and are dilatates of each other.
\end{cor}

\begin{proof}
The inequality is trivial if ${V}(K)={V}(L)=0$.  Otherwise, the result follows immediately from Corollary~\ref{OBMnewer} (or Corollary~\ref{corBrunnMinkowski2new}) with $\varphi(x_1,x_2)=x_1^p+x_2^p$.
\end{proof}

\section{Orlicz linear combination and an Orlicz mixed volume}\label{OMImv}

In this section we seek to calculate the first variation of volume with respect to Orlicz addition.  In other words, we require an appropriate generalization of the $L_p$ mixed volume
\begin{equation}\label{Lpvariation}
V_p(K,L)=\frac{p}{n}\lim_{\ee\to 0+}\frac{V(K+_p\ee^{1/p} L)-V(K)}{\ee}=\frac{1}{n}\int_{S^{n-1}}
h_L(u)^ph_K(u)^{1-p}\,dS(K,u),
\end{equation}
see \cite{L1}.  The quantity $\ee^{1/p}$ appears as a consequence of the definition of an $L_p$ scalar multiplication, denoted here by $\cdot_p$, via the equation $\alpha\cdot_p K=\alpha^{1/p}K$, for all $\alpha\ge 0$ and $K\in{\mathcal K}^n_{o}$.

Suppose that $\alpha_j\ge 0$ and $\varphi_j\in \Phi$, $j=1,\dots,m$.  If $K_j\in{\mathcal K}^n_{o}$, $j=1,\dots,m$, we define the {\em Orlicz linear combination} $+_{\varphi}(K_1,\dots,K_m,\alpha_1,\dots,\alpha_m)$ by
\begin{equation}\label{hK02}
h_{+_{\varphi}(K_1,\dots,K_m,\alpha_1,\dots,\alpha_m)
}(x)=\inf\left\{\lambda>0:\sum_{j=1}^m\alpha_j\,
\varphi_j\left(\frac{h_{K_j}(x)}
{\lambda}\right)\le 1\right\},
\end{equation}
for all $x\in \R^n$. Unlike the $L_p$ case, it is not generally possible to isolate an Orlicz scalar multiplication, since there is a dependence not just on one coefficient $\alpha_j$ but on all $\alpha_1,\dots,\alpha_m$ and $K_1,\ldots,K_m$.

\begin{rem}\label{remLC}
{\em  Definition (\ref{hK02}) corresponds to taking the function $\varphi$ in (\ref{OrlComb}) and (\ref{Orldef}) to be
$$
\varphi(x_1,\dots,x_m)=\sum_{j=1}^m\alpha_j\varphi_j(x_j).
$$
Note that in this case it is no longer true, as in (\ref{spphi}), that $\varphi\in \Phi_m$.  However, $\varphi\in \overline{\Phi}_m$ as defined in Section~\ref{prelim} if $\alpha_1,\ldots,\alpha_m>0$.
}
\end{rem}

For our purposes, it suffices to focus on the case $m=2$.  The Orlicz linear combination $+_{\varphi}(K,L,\alpha,\beta)$, for $K, L\in{\mathcal K}^n_{o}$ and $\alpha, \beta\ge 0$, can be defined equivalently via the implicit equation
\begin{equation}\label{OrldefGG}
\alpha\varphi_1\left( \frac{h_K(x)}{h_{+_{\varphi}(K,L,\alpha,\beta)}(x)}\right) + \beta{\varphi}_2\left( \frac{h_L(x)}{h_{+_{\varphi}(K,L,\alpha,\beta)}(x)}\right) = 1,
\end{equation}
if $\alpha h_K(x)+\beta h_L(x)>0$, and by $h_{+_{\varphi}(K,L,\alpha,\beta)}(x)=0$ if $\alpha h_K(x)+\beta h_L(x)=0$, for all $x\in \R^n$.

It is easy to verify that when $\varphi_1(t)=\varphi_2(t)=t^p$, $p\ge 1$, the Orlicz linear combination $+_{\varphi}(K,L,\alpha,\beta)$ equals the $L_p$ linear combination $\alpha\cdot_pK+_{p} \,\beta\cdot_pL$.
Setting $L=\{o\}$ in (\ref{OrldefGG}), we see that
$$
+_{\varphi}(K,\{o\},\alpha,\beta)=\left(\widehat{\varphi_1}^{-1}
\left(\frac{1}{\alpha}\right)\right)^{-1}K,
$$
and a similar relation is obtained for $K=\{o\}$ in (\ref{OrldefGG}).
Here, as before, $\widehat{\varphi_1}$ denotes the restriction of $\varphi_1$ to $[\tau,\infty)$, where $\tau=\tau(\varphi_1)<1$ is given by (\ref{tau}).

Henceforth we shall write  $K+_{\varphi,\ee}L$ instead of $+_{\varphi}(K,L,1,\ee)$, for $\ee\ge 0$, and assume throughout that this is defined by (\ref{OrldefGG}), where $\alpha=1$, $\beta=\ee$, and $\varphi_j\in \Phi$, $j=1,2$.

\begin{lem}\label{newconvlem}
Let $K,L\in {\mathcal{K}}_{o}^n$.  Then
$$K+_{\varphi,\ee}L\to K$$
in the Hausdorff metric as $\ee\to 0+$.
\end{lem}

\begin{proof}
Let $u\in S^{n-1}$. It follows from \eqref{OrldefGG} with $\alpha_1=1$ and $\alpha_2=\ee$ and the preceding remarks that
$$h_K(u)\le h_{K+_{\varphi,\ee}L}(u)\le h_{K+_{\varphi,1}L}(u),$$
for  all $\ee\in (0,1]$. If $h_K(u)>0$, we conclude from
$$
\varphi_1\left( \frac{h_K(u)}{h_{ K+_{\varphi,\ee}L}(u)}\right) + \ee\varphi_2\left( \frac{h_L(u)}{h_{K+_{\varphi,\ee}L}(u)}\right) = 1
$$
that if $\ee\to0+$ and if a subsequence of \{$h_{K+_{\varphi,\ee}L}(u):\ee\in (0,1]\}$ converges to a constant $k\ge h_K(u)$, then
$$
\varphi_1\left( \frac{h_K(u)}{k}\right)=\varphi_1\left( \frac{h_K(u)}{k}\right) + 0 \cdot\varphi_2\left( \frac{h_L(u)}{k}\right) = 1=\varphi_1(1)
$$
and hence $k=h_K(u)$. Therefore, in this case,  $h_{K+_{\varphi,\ee}L}(u)\to h_K(u)$ as $\ee\to 0+$.

If $h_K(u)=0$ and $h_L(u)>0$, then \eqref{OrldefGG} with $\alpha_1=1$ and $\alpha_2=\ee$ implies that
$$
\varphi_2\left( \frac{h_L(u)}{h_{K+_{\varphi,\ee}L}(u)}\right) = \frac{1}{\ee}.
$$
For $0<\ee<1/\tau_2\le \infty$, we get
$$
h_{K+_{\varphi,\ee}L}(u)=\left(\widehat{\varphi_2}^{-1}
\left(\frac{1}{\ee}\right)\right)^{-1}\, h_L(u),
$$
and thus again $h_{K+_{\varphi,\ee}L}(u)\to 0=h_K(u)$ as $\ee\to 0+$. The latter also holds if $h_K(u)=h_L(u)=0$.
This shows that $h_{K+_{\varphi,\ee}L}\to h_K$  as $\ee\to 0+$ holds pointwise, and thus also uniformly on $S^{n-1}$ (see \cite[Theorem~1.8.12]{Sch93}).
\end{proof}

\begin{lem}\label{Lutnew} Let $K,L\in {\mathcal{K}}_{o}^n$. Suppose that
$$\lim_{\ee\rightarrow 0+}\frac{h_{K+_{\varphi,\ee}L}(u)-h_K(u)}{\ee}=F_{K,L,\varphi_1, \varphi_2}(u),$$
uniformly for $u\in S^{n-1}$, where $F_{K,L,\varphi_1, \varphi_2}:S^{n-1}\rightarrow [0,\infty)$ is a measurable function.  Then
\begin{equation}\label{intrepr}
\lim_{\ee\rightarrow 0+}\frac{V(K+_{\varphi,\ee}L)-V(K)}{\ee}=\int_{S^{n-1}}
F_{K,L,\varphi_1, \varphi_2}(u)\,dS(K,u).
\end{equation}
\end{lem}

\begin{proof}
It would be possible to apply the argument of \cite[Theorem~1.1]{L1} (cf.~also the argument in \cite[Lemma  6.5.3]{Sch93}, which is attributed to Aleksandrov), but for the reader familiar with mixed volumes and mixed surface area measures, we provide an alternative proof which avoids the use of Minkowski's inequality.  Our notation follows that of \cite{Sch93}.  For brevity, we temporarily write $K_{\ee}=K+_{\varphi,\ee}L$.  Starting with the decomposition
$$
\frac{V(K_{\ee})-V(K)}{\ee}=\sum_{i=0}^{n-1}\frac{
V(K_\ee[i+1],K[n-1-i])-V(K_\ee[i],K[n-i])}{\ee},
$$
it is clearly sufficient to show that each of the $n$ summands converges to $I/n$ as $\ee\rightarrow 0+$, where $I$ is the integral on the right-hand side of (\ref{intrepr}). For this, observe that
\begin{eqnarray}
\lefteqn{\frac{V(K_\ee[i+1],K[n-1-i])-V(K_\ee[i],K[n-i])}{\ee}}\nonumber\\
&=&\frac{1}{n}\int_{S^{n-1}}\frac{h_{K_\ee}(u)-h_K(u)}{\ee}\, dS(K_\ee[i],K[n-1-i],u)\nonumber\\
&=&\frac{1}{n}\int_{S^{n-1}}\left(\frac{h_{K_\ee}(u)-h_K(u)}{\ee}
-F_{K,L,\varphi_1, \varphi_2}(u)\right)\, dS(K_\ee[i],K[n-1-i],u)\label{xyz} \\
& &\qquad \, + \frac{1}{n}\int_{S^{n-1}} F_{K,L,\varphi_1, \varphi_2}(u) \, dS(K_\ee[i],K[n-1-i],u).\nonumber
\end{eqnarray}
By assumption, the integrand in (\ref{xyz}) converges uniformly to zero for $u\in S^{n-1}$.  Since $K_\ee\to K$ as $\ee\to 0+$, by Lemma~\ref{newconvlem}, and the mixed surface area measures $S(K_\ee[i],K[n-1-i],\cdot)$ are uniformly bounded for $\ee\in (0,1]$, the first integral in the previous sum converges to zero.  Noting that $S(K_\ee[i],K[n-1-i],\cdot)\to S(K,\cdot)$ weakly as $\ee\to 0+$, we see that the second integral converges to $I/n$, as required.
\end{proof}

\begin{lem}\label{limiteqns} Let $K\in {\mathcal{K}}_{oo}^n$ and $L\in {\mathcal{K}}_{o}^n$.  Then
\begin{equation}\label{newlimits}
\lim_{\ee\rightarrow 0+}\frac{h_{K+_{\varphi,\ee}L}(u)-h_K(u)}{\ee}=\frac{h_K(u)}{(\varphi_1)'_l(1)}
\varphi_2\left(\frac{h_L(u)}{h_K(u)}\right),
\end{equation}
uniformly for $u\in S^{n-1}$.
\end{lem}

\begin{proof}
Let $\ee>0$, let $K\in {\mathcal{K}}_{oo}^n$ and $L\in {\mathcal{K}}_{o}^n$, and let $u\in S^{n-1}$.

If $h_L(u)/h_K(u)\le \tau=\tau(\varphi_2)$, then by considerations similar to those in Remark~\ref{Remark4.2}, we see that \eqref{OrldefGG} with $\alpha_1=1$ and $\alpha_2=\ee$ implies that $h_{K+_{\varphi,\ee}L}(u)=h_K(u)$
for all $\ee\ge 0$. Thus, in this case, the assertion of the lemma holds. If $h_L(u)/h_K(u)> \tau$, then
$h_L(u)/h_{K+_{\varphi,\ee}L}(u)>\tau$ for $\ee>0$ sufficiently small. Moreover,
we can assume that $h_K(u)/h_{K+_{\varphi,\ee}L}(u)>\tau$ by choosing $\ee>0$ sufficiently small.
In the following, we need only consider this case.
Then, by \eqref{OrldefGG} with $\alpha_1=1$ and $\alpha_2=\ee$, we have
\begin{equation}\label{ddnn}
\frac{h_{K+_{\varphi,\ee}L}(u)-h_K(u)}{h_{K+_{\varphi,\ee}L}(u)}= 1-\frac{h_K(u)}{h_{K+_{\varphi,\ee}L}(u)}= 1-\widehat{\varphi_1}^{-1}\left(1-\ee\varphi_2\left(\frac{ h_L(u)}{
h_{K+_{\varphi,\ee}L}(u)}\right)\right).
\end{equation}
Hence, using (\ref{ddnn}) and Lemma~\ref{newconvlem}, we obtain
\begin{eqnarray}\label{nww1}
\lefteqn{\lim_{\ee\rightarrow 0+}\frac{h_{K+_{\varphi,\ee}L}(u)-h_K(u)}{\ee}}\nonumber\\
& = &
\lim_{\ee\rightarrow 0+}\frac{h_{K+_{\varphi,\ee} L}(u)}{\ee}
\left(1-\widehat{\varphi_1}^{-1}\left(1-\ee\varphi_2\left(\frac{h_L(u)}{
h_{K+_{\varphi,\ee}L}(u)}\right)\right)\right)\nonumber\\
& = &\lim_{\ee\rightarrow 0+}h_{K+_{\varphi,\ee}L}(u)\varphi_2\left(\frac{ h_L(u)}{
h_{K+_{\varphi,\ee}L}(u)}\right)
\left(\frac{1-\widehat{\varphi_1}^{-1}\left(1-\ee\varphi_2\left(\frac{h_L(u)}{
h_{K+_{\varphi,\ee}L}(u)}\right)\right)}{
1-\left(1-\ee\varphi_2\left(\frac{h_L(u)}{
h_{K+_{\varphi,\ee}L}(u)}\right)\right)}\right)\nonumber\\
& = &h_{K}(u)\varphi_2\left(\frac{ h_L(u)}{
h_{K}(u)}\right)\lim_{\ee\rightarrow 0+}
\left(\frac{1-\widehat{\varphi_1}^{-1}\left(1-\ee\varphi_2\left(\frac{h_L(u)}{
h_{K+_{\varphi,\ee}L}(u)}\right)\right)}{
1-\left(1-\ee\varphi_2\left(\frac{h_L(u)}{
h_{K+_{\varphi,\ee}L}(u)}\right)\right)}\right).
\end{eqnarray}
Let
\begin{equation}\label{zzzz}
z=\widehat{\varphi_1}^{-1}\left(1-\ee\varphi_2\left(\frac{h_L(u)}{
h_{K+_{\varphi,\ee}L}(u)}\right)\right)
\end{equation}
and note that $z\rightarrow 1-$ as $\ee\rightarrow 0+$.  Consequently,
\begin{equation}\label{nww2}
\lim_{\ee\rightarrow 0+}
\left(\frac{1-\widehat{\varphi_1}^{-1}\left(1-\ee\varphi_2\left(\frac{h_L(u)}{
h_{K+_{\varphi,\ee}L}(u)}\right)\right)}{
1-\left(1-\ee\varphi_2\left(\frac{h_L(u)}{
h_{K+_{\varphi,\ee}L}(u)}\right)\right)}\right)= \lim_{z\rightarrow 1-}
\frac{1-z}{\varphi_1(1)-\varphi_1(z)}=\frac{1}{(\varphi_1)'_l(1)}.
\end{equation}
Now the pointwise limit (\ref{newlimits}) follows immediately from (\ref{nww1}) and (\ref{nww2}).

Moreover, the convergence is uniform for $u\in S^{n-1}$.  Indeed, by (\ref{nww1}), (\ref{zzzz}), and (\ref{nww2}), it suffices to recall that by Lemma~\ref{newconvlem},
$$\lim_{\ee\rightarrow 0+}h_{K+_{\varphi,\ee}L}(u)=h_K(u),$$
uniformly for $u\in S^{n-1}$.
\end{proof}

For $\varphi\in \Phi$, define the Orlicz mixed volume $V_{\varphi}(K,L)$ by
\begin{equation}\label{intreprLp}
V_{\varphi}(K,L)=\frac{1}{n}\int_{S^{n-1}}
\varphi\left(\frac{h_L(u)}{h_K(u)}\right)h_K(u)\,dS(K,u),
\end{equation}
for all $K\in {\mathcal{K}}_{oo}^n$ and $L\in {\mathcal{K}}_{o}^n$.  The following result links the Orlicz mixed volume to the first variation of volume with respect to Orlicz addition and scalar multiplication.

\begin{thm}\label{OrlMV}
For all $K\in {\mathcal{K}}_{oo}^n$ and $L\in {\mathcal{K}}_{o}^n$, we have
\begin{equation}\label{OrlMVdefff}
V_{\varphi_2}(K,L)=\frac{(\varphi_1)'_l(1)}{n}\lim_{\ee\rightarrow 0+}\frac{V(K+_{\varphi,\ee}L)-V(K)}{\ee}.
\end{equation}
\end{thm}

\begin{proof}
This follows immediately from Lemmas~\ref{Lutnew} and~\ref{limiteqns}.
\end{proof}

When $\varphi_1(t)=\varphi_2(t)=t^p$, $p \ge 1$, (\ref{intreprLp}) and (\ref{OrlMVdefff}) reduce to their $L_p$ versions in (\ref{Lpvariation}).

\section{An Orlicz-Minkowski inequality}\label{OMInequality}

Most of the work in this section is done by the following lemma.

\begin{lem}\label{Dan}
Let $0<a\le \infty$ be an extended real number, and let $I=[0,a)$ be a possibly infinite interval. Suppose that $\varphi : I\to [0,\infty )$ is convex with $\varphi (0)=0$. If $K\in {\mathcal{K}}_{oo}^n$ and $L\in {\mathcal{K}}_{o}^n$ are such that $L\subset \inte(aK)$, then
\begin{equation}\label{Jen}
\frac{1}{nV(K)}\int_{S^{n-1}}\varphi\left(\frac{h_L(u)}{h_K(u)}\right)h_K(u)\,
dS(K,u)\ge \varphi\left(\left(\frac{V(L)}{V(K)}\right)^{1/n}\right).
\end{equation}
If $\varphi$ is strictly convex, equality holds if and only if $K$ and $L$ are dilatates or $L=\{o\}$.
\end{lem}

\begin{proof}
Note that if $L\subset\inte( aK)$, we have $h_L(u)/h_K(u)\in I$ for all $u\in S^{n-1}$.  Since the normalized cone measure $\overline{V}_n(K,\cdot)$ defined by (\ref{normconem}) is a probability measure on $S^{n-1}$, we can use Jensen's inequality (\ref{JenIneq}), Minkowski's first inequality (\ref{MinkIneq}), and the fact that $\varphi$ is increasing, to obtain
\begin{eqnarray*}
\frac{1}{nV(K)}\int_{S^{n-1}}\varphi\left(\frac{h_L(u)}{h_K(u)}\right)h_K(u)\,
dS(K,u)&\ge & \varphi\left(\frac{1}{nV(K)}
\int_{S^{n-1}}h_L(u)\,dS(K,u)\right)\\
&=&\varphi\left(\frac{V_1(K,L)}{V(K)}\right)\\
&\ge & \varphi\left(\left(\frac{V(L)}{V(K)}\right)^{1/n}\right).
\end{eqnarray*}
Suppose that equality holds in (\ref{Jen}) and $\varphi$ is strictly convex, so that $\varphi>0$ on $(0,a)$.  Then in view of the injectivity of $\varphi$, we have equality in Minkowski's first inequality, so there are $r\ge 0$ and $x\in \R^n$ such that $L=rK+x$ and hence
\begin{equation}\label{lksupp}
h_L(u)=rh_K(u)+x\cdot u,
\end{equation}
for all $u\in S^{n-1}$.  Since equality must hold in Jensen's inequality as well, when $\varphi$ is strictly convex we can conclude from the equality condition for Jensen's inequality that
$$
\frac{1}{nV(K)}\int_{S^{n-1}} \frac{h_L(u)}{h_K(u)}h_K(u)\,dS(K,u)=\frac{h_L(v)}{h_K(v)},
$$
for $S(K,\cdot)$-almost all $v\in S^{n-1}$. Substituting (\ref{lksupp}), we obtain
$$
\frac{1}{nV(K)}\int_{S^{n-1}} \left(r+\frac{x\cdot u}{h_K(u)}\right)h_K(u)\,dS(K,u)=r+\frac{x\cdot v}{h_K(v)},
$$
for $S(K,\cdot)$-almost all $v\in S^{n-1}$. From this and the fact that the centroid of $S(K,\cdot)$ is at the origin, we get
$$
0=x\cdot \left(\frac{1}{nV(K)}\int_{S^{n-1}}u\,dS(K,u)\right)=
\frac{1}{nV(K)}\int_{S^{n-1}}x\cdot u\,dS(K,u)=\frac{x\cdot v}{h_K(v)},
$$
that is, $x\cdot v=0$, for $S(K,\cdot)$-almost all $v\in S^{n-1}$. Since $K\in {\mathcal{K}}_{oo}^n$, the support of
$S(K,\cdot)$ is not contained in a great subsphere, which implies that $x=o$ and hence $L=rK$.
\end{proof}

When $a=\infty$ and $\varphi(t)=e^t-1$, for example, (\ref{Jen}) can be presented in the appealing form
$$\log\int_{S^{n-1}}\exp\left(\frac{h_L(u)}{h_K(u)}\right)\,d\overline{V}_n(K,u)\ge
\left(\frac{V(L)}{V(K)}\right)^{1/n}.$$
(Similarly, (\ref{Lpvariation22}) can be written as
$$\left(\int_{S^{n-1}}\left(\frac{h_L(u)}{h_K(u)}\right)^p\,d\overline{V}_n(K,u)
\right)^{1/p}\ge \left(\frac{V(L)}{V(K)}\right)^{1/n}.$$
The left-hand side is just the $p$th mean of the function $h_L/h_K$ with respect to $\overline{V}_n(K,\cdot)$; since $p$th means increase with $p$, Minkowski's first inequality (\ref{MinkIneq}) implies (\ref{Lpvariation22}) for all $p>1$.)

The following result provides an Orlicz-Minkowski inequality which in the special case when $\varphi(t)=t^p$, $p\ge 1$, reduces to the $L_p$-Minkowski inequality (\ref{Lpvariation22}).

\begin{thm}\label{NewMIthm}
Let $\varphi\in \Phi$.  If $K\in {\mathcal{K}}_{oo}^n$ and $L\in {\mathcal{K}}_{o}^n$, then
\begin{equation}\label{NewMI}
V_{\varphi}(K,L)\ge V(K)\varphi\left(\left(\frac{V(L)}{V(K)}\right)^{1/n}\right).
\end{equation}
If $\varphi$ is strictly convex, equality holds if and only if $K$ and $L$ are dilatates or $L=\{o\}$.
\end{thm}

\begin{proof}
This follows immediately from (\ref{intreprLp}) and Lemma~\ref{Dan} with $a=\infty$.
\end{proof}

Inequality (\ref{NewMI}) and its equality condition implies the Orlicz-Brunn-Minkowski inequality (\ref{OBMnn}) when $\varphi(x_1,x_2)=\varphi_1(x_1)+\varphi_2(x_2)$ for $x_1,x_2\ge 0$ and some $\varphi_1,\varphi_2\in \Phi$, $K\in {\mathcal{K}}_{oo}^n$, and $L\in {\mathcal{K}}_{o}^n$, and its equality condition.  Indeed, by (\ref{Orldef}) and by (\ref{intreprLp}) and (\ref{NewMI}) with $K$, $L$, and $\varphi$ replaced by $K+_{{{\varphi}}}L$, $K$, and $\varphi_1$ (and by $K+_{{{\varphi}}}L$, $L$, and $\varphi_2$), respectively, we have
\begin{eqnarray*}
V(K+_{{{\varphi}}}L)
&=&\frac{1}{n}\int_{S^{n-1}}\left(\varphi_1\left(\frac{h_{K}(u)}
{h_{K+_{{{\varphi}}}L}(u)}\right)+\varphi_2\left(\frac{h_{L}(u)}
{h_{K+_{{{\varphi}}}L}(u)}\right)\right)h_{K+_{{{\varphi}}}L}(u)\,dS(K+_{{{\varphi}}}L,u)\\
&=&V_{\varphi_1}(K+_{{{\varphi}}}L,K)+V_{\varphi_2}(K+_{{{\varphi}}}L,L)\\
&\ge &V(K+_{{{\varphi}}}L)\left(
\varphi_1\left(\left(\frac{V(K)}{V(K+_{{{\varphi}}}L)}\right)^{1/n}\right)
+\varphi_2\left(\left(\frac{V(L)}{V(K+_{{{\varphi}}}L)}\right)^{1/n}\right)\right).
\end{eqnarray*}
This is just (\ref{OBMnn}). If  equality holds in (\ref{OBMnn}), then it also holds in the previous inequality and therefore in (\ref{NewMI}), with $K$, $L$, and $\varphi$ replaced by $K+_{{{\varphi}}}L$, $K$, and $\varphi_1$ (and by $K+_{{{\varphi}}}L$, $L$, and $\varphi_2$), respectively.  Thus if $\varphi$ is strictly convex, then $\varphi_1$ and $\varphi_2$ are also, so both $K$ and $L$ are multiples of $K+_{{{\varphi}}}L$, and hence are dilatates of each other or $L=\{o\}$.

\section{Other Orlicz operations and Orlicz inequalities}\label{Alternative}

Unless otherwise specified, it will be assumed throughout this section that $\varphi\in \Phi=\Phi_1$ and $\varphi>0$ on $(0,\infty)$.

Perhaps the most natural way to attempt to define an Orlicz addition is by setting
\begin{equation}\label{diffdef}
h_{K*_\varphi L}(x)={\varphi}^{-1}\left(\varphi\left(h_{K}(x)\right)+
\varphi\left(h_{L}(x)\right)
\right),
\end{equation}
for all $K,L\in{\mathcal K}^n_{o}$ and $x\in \R^n$.  As this operation is obviously associative, it may seem more appealing than the definition via (\ref{Orldef}). However, the following theorem shows that it yields nothing new beyond $L^p$ addition, and in fact (\ref{diffdef}) reduces to (\ref{Lpaddition}).

\begin{thm}\label{OrlDef2}
Equation (\ref{diffdef}) implies that $\varphi(t)=t^p$ for some $p\ge 1$.
\end{thm}

\begin{proof}
Using (\ref{hproj}), it is easy to show that the operation $*_{\varphi}$ defined by (\ref{diffdef}) is projection covariant.  Since it is clearly also associative and has the identity property, it follows from \cite[Theorem~7.9]{GHW} that as an operation between $o$-symmetric compact convex sets, it is $L_p$ addition for some $1\le p\le \infty$.  If $p<\infty$, we then deduce from (\ref{diffdef}) with $K=rB^n$ and $L=sB^n$ for some $r,s\ge 0$ that
\begin{equation}\label{prs}
\left(r^p+s^p\right)^{1/p}=h_{rB^n+_p sB^n}(u)=h_{rB^n*_\varphi sB^n}(u)=
{\varphi}^{-1}\left(\varphi(r)+\varphi(s)\right),
\end{equation}
for all $u\in S^{n-1}$.  Setting $g(t)=\varphi\left(t^{1/p}\right)$ for $t\ge 0$, (\ref{prs}) becomes
$$g\left(r^p+s^p\right)=g(r^p)+g(s^p).$$
Thus $g(x+y)=g(x)+g(y)$ for all $x,y\ge 0$ and $g$ satisfies Cauchy's functional equation on $[0,\infty)$.  Since $g$ is continuous, nonnegative, and not identically zero, we have $g(t)=ct$ for some $c>0$. (See, for example, \cite[Theorem~1, p.~34]{Acz66}.)  Since $\varphi(1)=1$, this yields $\varphi(t)=t^p$.

If $p=\infty$, then instead of (\ref{prs}), we get
$$\max\{r,s\}=h_{rB^n+_{\infty} sB^n}(u)=h_{rB^n*_\varphi sB^n}(u)={\varphi}^{-1}\left(\varphi(r)+\varphi(s)\right),$$
for all $u\in S^{n-1}$.  This implies that for all $r\ge s\ge 0$, $\varphi(r)=\varphi(r)+\varphi(s)$, so $\varphi\equiv 0$, a contradiction.
\end{proof}

There is also a more immediate generalization of the $L_p$ scalar multiplication than that introduced in Section~\ref{OMImv}, namely via the equation $\alpha\cdot_{\varphi} K={\varphi}^{-1}(\alpha)K$, for all $\alpha\ge 0$ and $K\in{\mathcal K}^n_{o}$.  However, the following result shows that this also does not lead to anything beyond the $L_p$ case, even when combined with the Orlicz addition defined in Section~\ref{Orlicz}.

\begin{thm}\label{OrlMVdeadend}
Let $\varphi\in \Phi$ satisfy $\varphi>0$ on $(0,\infty)$. The limit
$$\lim_{\ee\rightarrow 0+}\frac{V(K+_\varphi \varphi^{-1}(\ee) L)-V(K)}{\ee}$$
exists for all $K,L\in {\mathcal{K}}_{oo}^n$ only if the function $f_{\varphi}$ defined by \begin{equation}\label{flim}
f_{\varphi}(t)=\lim_{\ee\rightarrow 0+}
\frac{\varphi(\ee t)}{\varphi(\ee)}
\end{equation}
is a real-valued function on $[0,\infty)$.  In this case, $f_{\varphi}(t)=t^p$ for some $p=p(\varphi)\ge 1$, and for this $p$, we have
$$
\lim_{\ee\rightarrow 0+}\frac{V(K+_\varphi \varphi^{-1}(\ee) L)-V(K)}{\ee}=\frac{1}{\varphi'_l(1)}\int_{S^{n-1}}
h_L(u)^ph_K(u)^{1-p}\,dS(K,u),
$$
for all $K\in {\mathcal{K}}_{oo}^n$ and $L\in {\mathcal{K}}_{o}^n$.
\end{thm}

We omit the proof of this theorem, which can be obtained by the methods of Section~\ref{OMImv}.  It turns out that instead of the right-hand side of (\ref{newlimits}), one arrives at
$$
\lim_{\ee\rightarrow 0+}\frac{h_{K+_\varphi \varphi^{-1}(\ee) L}(u)-h_K(u)}{\ee}=\frac{h_K(u)}{\varphi'_l(1)}\lim_{\ee\rightarrow 0+}
\frac{\varphi\left(\ee\frac{h_L(u)}{h_K(u)}\right)}{\varphi(\ee)},
$$
for $u\in S^{n-1}$, when either limit exists.  The crucial observation is then that
if (\ref{flim}) defines a real-valued function $f_{\varphi}$ on $[0,\infty)$ (i.e., the limit in \eqref{flim} exists and is finite for all $t\ge 0$), then $f_{\varphi}$ satisfies the multiplicative Cauchy functional equation $f_{\varphi}(st)=f_{\varphi}(s)f_{\varphi}(t)$ for $s,t\ge 0$.  From this one concludes that $f_{\varphi}(t)=t^p$ for some $p=p(\varphi)\ge 1$.

There are functions $\varphi\in \Phi$ such that $\varphi>0$ on $(0,\infty)$ for which $f_{\varphi}$ in (\ref{flim}) is not a real-valued function on $[0,\infty)$.  To see this, note that the function $e^{4(1-t^{-2})}$ tends to zero as $t\rightarrow 0$, is convex and strictly increasing on $[0,1]$, and has slope 8 at $t=1$.  Therefore the function
$$
\varphi(t)=\begin{cases}
0,& {\text{ if $t=0$,}}\\
e^{4(1-t^{-2})},& {\text{ if $0<t\le 1$},}\\
8t-7,& {\text{ if $t\ge 1$,}}
\end{cases}
$$
satisfies the conditions we are assuming.  Then $f_{\varphi}(0)=0$ and
$$
f_{\varphi}(t)=\lim_{\ee\rightarrow 0+}\frac{e^{4(1-(\ee t)^{-2})}}{e^{4(1-\ee^{-2})}}=\lim_{\ee\rightarrow 0+}e^{4\ee^{-2}(1-t^{-2})}=\begin{cases}
0,& {\text{ if $0<t<1$,}}\\
1,& {\text{ if $t=1$},}\\
\infty,& {\text{ if $t>1$.}}
\end{cases}
$$
Thus in this case, Theorem~\ref{OrlMVdeadend} does not apply.

In view of Theorem~\ref{OrlDef2}, if $\varphi(t)\neq t^p$, then there are $K$ and $L$ such that (\ref{diffdef}) does not define a support function.  This obstacle can be avoided by defining, for general $\varphi\in \Phi$,
$$
K\widehat{+}_{\varphi}L=\bigcap_{u\in S^{n-1}}\left\{x\in \R^n: x\cdot u\le
\widehat{\varphi}^{-1}\left(\varphi\left(h_{K}(u)\right)+\varphi\left(h_{L}(u)\right)
\right)\right\},
$$
for $K\in\K^n_{oo}$ and $L\in\K^n_o$.
Then $K\widehat{+}_\varphi L$ is the Wulff shape (also sometimes called the Aleksandrov body) of the function on the right-hand side of (\ref{diffdef}); see, for example, \cite[Section~6]{Gar02}.

If we also define a scalar multiplication by setting
\begin{equation}\label{Wulphis}
K\widehat{+}_{\varphi}\,\ee\cdot L=\bigcap_{u\in S^{n-1}}\left\{x\in \R^n: x\cdot u\le
\widehat{\varphi}^{-1}\left(\varphi\left(h_{K}(u)\right)+\ee \varphi\left(h_{L}(u)\right)
\right)\right\},
\end{equation}
for  $K\in\K^n_{oo}$ and $L\in\K^n_o$,
then we can compute a first variation of volume if also $\varphi>0$ on $(0,\infty)$. Using (\ref{Wulphis}) and \cite[Lemma~1]{HLYZ}, with $h_t(u)=\varphi^{-1}\left(\varphi\left(h_{K}(u)\right)+t \varphi\left(h_{L}(u)\right)
\right)$, for all $u\in S^{n-1}$, we obtain
$$
\lim_{\ee\rightarrow 0+}\frac{V(K\widehat{+}_\varphi \,\ee \cdot L)-V(K)}{\ee}=\int_{S^{n-1}}
\frac{\varphi\left(h_L(u)\right)}{\varphi_r'\left(h_K(u)\right)}\,dS(K,u).
$$
However, it is not clear what Minkowski-type inequality would be satisfied by this first variation of volume.

The definition (\ref{Orlprojsupp}) of the Orlicz projection body suggests defining, by analogy,
\begin{equation}\label{FV}
\widehat{V_{\varphi}}(K,L)=\inf\left\{\lambda>0:\int_{S^{n-1}}\varphi
\left(\frac{h_L(u)}{\lambda h_K(u)}\right)\,d\overline{V}_n(K,u)\le 1\right\},
\end{equation}
for $\varphi\in\Phi$, $K\in\K^n_{oo}$, and $L\in\K^n_o$, where $\overline{V}_n(K,\cdot)$ is the normalized cone measure for $K$ defined by (\ref{normconem}). Note however that if $\varphi(t)=t^p$, $p\ge 1$, we have $\widehat{V_{\varphi}}(K,L)=\left(V_p(K,L)/V(K)\right)^{1/p}$ (instead of $V_p(K,L)$).

\begin{thm}\label{FVnew}
Let $\varphi\in \Phi$ and let $K\in {\mathcal{K}}_{oo}^n$ and $L\in {\mathcal{K}}_{o}^n$. Then
\begin{equation}\label{FVO}
\widehat{V_{\varphi}}(K,L)\ge \left(\frac{V(L)}{V(K)}\right)^{1/n}.
\end{equation}
If $\varphi$ is strictly convex and $V(L)>0$, then equality holds if and only if $K$ and $L$ are dilatates.
\end{thm}

\begin{proof}
For any $\lambda>0$, Lemma~\ref{Dan}, with $a=\infty$ and $K$ replaced by $\lambda K$, shows that
\begin{equation}\label{Jennn}
\int_{S^{n-1}}\varphi\left(\frac{h_L(u)}{\lambda h_K(u)}\right)\,
d\overline{V}_n(K,u)\ge \varphi\left(\frac{1}{\lambda}\left(\frac{V(L)}{V(K)}\right)^{1/n}\right).
\end{equation}
Let $\lambda>0$ be such that
$$\int_{S^{n-1}}\varphi
\left(\frac{h_L(u)}{\lambda h_K(u)}\right)\,d\overline{V}_n(K,u)\le 1.$$
Using (\ref{Jennn}), we obtain
$$1\ge \varphi\left(\frac{1}{\lambda}\left(\frac{V(L)}{V(K)}\right)^{1/n}\right),$$
or, since $\varphi(1)=1$,
$$\lambda\ge \left(\frac{V(L)}{V(K)}\right)^{1/n}.$$
In view of (\ref{FV}), this yields (\ref{FVO}).

Suppose that equality holds in \eqref{FVO}, $\varphi$ is strictly convex, and $V(L)>0$.
By (\ref{FV}),  $\lambda_0= \widehat{V_{\varphi}}(K,L)>0$
satisfies
$$\int_{S^{n-1}}\varphi\left(\frac{h_L(u)}{\lambda_0 h_K(u)}\right)\,
d\overline{V}_n(K,u)=1.
$$
By assumption, we have
$$
\lambda_0=\left(\frac{V(L)}{V(K)}\right)^{1/n},
$$
that is,
$$1=\varphi\left(\frac {1}{\lambda_0}\left(\frac{V(L)}{V(K)}\right)^{1/n}\right).$$
Hence equality holds in \eqref{Jennn} for $\lambda=\lambda_0$. But then the equality condition in Lemma~\ref{Dan} shows that $\lambda_0K$ and $L$ are dilatates.
\end{proof}

When $\varphi(t)=t^p$, (\ref{FVO}) also becomes the $L_p$-Minkowski inequality (\ref{Lpvariation22}), so it too can be viewed as an Orlicz-Minkowski inequality.

There is no direct relationship between the Orlicz-Minkowski inequalities (\ref{NewMI}) and (\ref{FVO}).  Indeed, when $\varphi>0$ on $(0,\infty)$, these can be written in the forms
$$
\frac{V_\varphi(K,L)}{V(K)}\ge \varphi\left(\left(\frac{V(L)}{V(K)}\right)^{\frac{1}{n}}\right)
$$
and
$$
\varphi\left(\widehat{V}_\varphi(K,L)\right)\ge \varphi\left(\left(\frac{V(L)}{V(K)}\right)
^{\frac{1}{n}}\right),
$$
respectively, and each of the two quantities on the left-hand sides can be larger than the other. This can be seen by taking $n=2$, $K$ a rectangle
with side lengths $a$ and $b$, and $L=B^2$. It is easy to check that in this case
$$
\frac{V_\varphi(K,L)}{V(K)}=\frac{1}{2}\left(\varphi\left(\frac{2}{a}\right)+
\varphi\left(\frac{2}{b}\right)\right)
$$
and $\lambda_1=\widehat{V}_\varphi(K,L)$ is determined by the equation
$$
\frac{1}{2}\left(\varphi\left(\frac{2}{\lambda_1 a}\right)+\varphi\left(\frac{2}{\lambda_1 b}\right)\right)=1.
$$
Choosing $\varphi(t)=(e^t-1)/(e-1)$, it now follows that
$$
\varphi(\lambda_1)\ge \frac{V_\varphi(K,L)}{V(K)}=\frac{e^{2/a}+e^{2/b}-2}{2(e-1)}
$$
if and only if
$$
H(a,b)=g_{a,b}\left(\ln\left(\frac{e^{2/a}+e^{2/b}}{2}\right)\right)-2e\ge 0,
$$
where
$$
g_{a,b}(t)=e^{2/(at)}+e^{2/(bt)}.
$$
Numerical calculations show that, for example, $H(2,1)<0$ while $H(2,b)>0$ if $b$ is sufficiently large.

\begin{thm}\label{Wolfgang1}
If $K\in {\mathcal{K}}_{oo}^n$ and $L\in {\mathcal{K}}_{o}^n$ are such that $L\subset \inte K$, then
\begin{equation}\label{WO1}
\log\left(\frac{V(K)^{1/n}-V(L)^{1/n}}{V(K)^{1/n}}
\right)\ge \int_{S^{n-1}}\log\left(\frac{h_K(u)-h_L(u)}{h_K(u)}\right)\,
d\overline{V}_n(K,u).
\end{equation}
Equality holds if and only if $K$ and $L$ are dilatates or $L=\{o\}$.
\end{thm}

\begin{proof}
Suppose that $K\in {\mathcal{K}}_{oo}^n$ and $L\in {\mathcal{K}}_{o}^n$ are such that $L\subset \inte K$.  Let $\varphi(t)=-\log(1-t)$. Then $\varphi(0)=0$ and $\varphi$ is strictly increasing and strictly convex on $[0,1)$ with $\varphi(t)\to\infty$ as
$t\to 1-$, and (\ref{WO1}) is a direct consequence of Lemma~\ref{Dan} with this choice of $\varphi$ and $a=1$.
\end{proof}

\begin{cor}\label{corWolfgang1}
If $K\in {\mathcal{K}}_{oo}^n$ and $L\in {\mathcal{K}}_{o}^n$, then
\begin{equation}\label{cons1}
\log\left(\frac{V(K+L)^{1/n}-V(L)^{1/n}}{V(K+L)^{1/n}}\right)\ge \int_{S^{n-1}}\log\left(\frac{h_K(u)}{h_{K+L}(u)}\right)\,
d\overline{V}_n(K+L,u).
\end{equation}
Equality holds if and only if $K$ and $L$ are dilatates or $L=\{o\}$.
\end{cor}

\begin{proof}
Since $K\in {\mathcal{K}}_{oo}^n$, we have $K+L\in {\mathcal{K}}_{oo}^n$ and $L\subset\inte(K+L)$.  Then (\ref{cons1}) follows immediately from Theorem~\ref{Wolfgang1} with $K$ replaced by $K+L$.
\end{proof}

For $K, L\in {\mathcal{K}}_{oo}^n$ and $0\le t\le 1$, define
$$
(1-t)\cdot K+_0 t\cdot L=\bigcap_{u\in S^{n-1}}\left\{x\in \R^n: x\cdot u\le
h_{K}(x)^{1-t}h_{L}(x)^{t}\right\}.
$$
B\"{o}r\"{o}czky, Lutwak, Yang, and Zhang \cite{BLYZ} conjecture that for $o$-symmetric convex bodies $K$ and $L$ in $\R^n$ and $0\le t\le 1$,
\begin{equation}\label{logBM}
V\left((1-t)\cdot K+_0 t\cdot L\right)\ge V(K)^{1-t}V(K)^t.
\end{equation}
They call (\ref{logBM}) the {\em log-Brunn-Minkowski inequality} and note that while it is not true for general convex bodies, it implies the classical Brunn-Minkowski inequality for $o$-symmetric convex bodies.  In \cite{BLYZ}, (\ref{logBM}) is proved when $n=2$, and it is also shown that for all $n$,  (\ref{logBM}) is equivalent to the {\em log-Minkowski inequality}
\begin{equation}\label{logMink}
\int_{S^{n-1}}\log\left(\frac{h_L(v)}{h_K(v)}\right)\,d\overline{V}_n(K,v)\ge \frac{1}{n}\log\left(\frac{V(L)}{V(K)}\right).
\end{equation}
Notice that this is just (\ref{Jen}) with the concave function $\varphi(t)=\log t$.  When $K$ and $L$ are replaced by $K+L$ and $K$, respectively, (\ref{logMink}) becomes
\begin{equation}\label{cons2}
\int_{S^{n-1}}\log\left(\frac{h_K(u)}{h_{K+L}(u)}\right)\,
d\overline{V}_n(K+L,u)\ge
\log\left(\left(\frac{V(K)}{V(K+L)}\right)^{\frac{1}{n}}\right).
\end{equation}
Combining \eqref{cons1} and \eqref{cons2}, we get
$$
V(K+L)^{1/n}-V(L)^{1/n}\ge V(K)^{1/n},
$$
whenever $K\in\mathcal{K}^n_{oo}$ and $L\in {\mathcal{K}}_{o}^n$ and (\ref{logMink}) holds with $K$ and $L$ replaced by $K+L$ and $K$, respectively.  In particular, if \eqref{logMink} holds (as it does, for $o$-symmetric convex bodies when $n=2$), then (\ref{cons1}) and \eqref{logMink} together split the classical Brunn-Minkowski inequality.

\end{document}